\documentclass[12pt, a4paper,leqno,oneside]{amschanged}

\usepackage{hyperref}

\usepackage{amsmath}
\usepackage{amssymb}
\usepackage{amsthm}

\usepackage{atbeginend}

\usepackage{psfrag}
\usepackage{graphicx}
\usepackage{hyperref}
\hypersetup{
    colorlinks,%
    citecolor=black,%
    filecolor=black,%
    linkcolor=black,%
    urlcolor=black
}

\usepackage{wasysym}
\usepackage{fullpage}

\usepackage{type1cm}

\newtheorem{theorem}{Theorem}[section]

\newtheorem{proposition}[theorem]{Proposition}
\newtheorem{lemma}[theorem]{Lemma}

\numberwithin{equation}{section}

\theoremstyle{remark}
\newtheorem{remark}[theorem]{Remark}

\def\supp{\mathop{\rm supp}\nolimits}
\def\dist{\mathop{\rm dist}\nolimits}

\def\hex{\hexagon}

\usepackage{titlesec}
\titleformat{\section}{\Large\bfseries}{\thesection}{1em}{}
\titleformat{\subsection}{\bfseries}{\thesubsection}{1em}{}

\renewcommand{\theenumi}{\arabic{enumi}}
\renewcommand{\labelenumi}{\theenumi)}

\usepackage{courier}

\usepackage{array}
\newcolumntype{e}{>{\displaystyle}r @{\,} >{\displaystyle}c @{\,} >{\displaystyle}l}

  \newcounter{constant}
  \newcommand{\newconstant}[1]{\refstepcounter{constant}\label{#1}}
  \newcommand{\useconstant}[1]{c_{\textnormal{\tiny \ref{#1}}}}
\def\clap#1{\hbox to 0pt{\hss#1\hss}}

\def\mathclap{\mathpalette\mathclapinternal}

\def\mathclapinternal#1#2{\clap{$\mathsurround=0pt#1{#2}$}}

\begin{document}

\fontsize{12}{14}\rm
\addtolength{\abovedisplayskip}{.5mm}
\addtolength{\belowdisplayskip}{.5mm}
\AfterBegin{enumerate}{\addtolength{\itemsep}{2mm}}



\title{\LARGE \usefont{T1}{tnr}{b}{n} \selectfont C\lowercase{ylinders' percolation in three dimensions}} 

\author{\normalsize \itshape M. H\lowercase{il\'ario} $^1$ \color{white} \tiny and}
\address{$^1$ Universidade Federal de Minas Gerais, Departamento de Matem\'atica, \newline \hspace*{10mm} Belo Horizonte 31270-901, Brazil {\itshape \texttt{mhilario@mat.ufmg.br}}.}

\author{\color{black} \normalsize \itshape V. S\lowercase{idoravicius} $^2$ \color{white} \tiny}
\address{$^2$ Instituto Nacional de Matem\'atica Pura e Aplicada, \newline \hspace*{10mm} Rio de Janeiro 22460-320, Brazil.}

\author{\color{black} \normalsize \itshape A. T\lowercase{eixeira} $^2$ $^3$}
\address{$^3$ \'Ecole Normale Sup\'erieure, D\'epartement de Math\'ematiques et Applications, \newline \hspace*{10mm} Paris 75230, France {\itshape \texttt{augusto.teixeira@ens.fr}}.}

\date{\today}

\begin{abstract}
We study the complementary set of a Poissonian ensemble of infinite cylinders in $\mathbb{R}^3$, for which an intensity parameter $u > 0$ controls the amount of cylinders to be removed from the ambient space.
We establish a non-trivial phase transition, for the existence of an unbounded connected component of this set, as $u$ crosses a critical non-degenerate intensity $u_*$.
We moreover show that this complementary set percolates in a sufficiently thick slab, in spite of the fact that it does not percolate in any given plane of $\mathbb{R}^3$, regardless of the choice of $u$.
\end{abstract}

\maketitle

\section{Introduction}
\label{s:intro}

In this article we study percolation on the subset obtained by removing from $\mathbb{R}^3$ a Poissonian ensemble of infinite cylinders of radius one.
Before presenting our main results, let us give some of the motivation and historical background to this problem.


Perhaps the simplest model for a random environment in $\mathbb{R}^d$ is the so-called ``continuum'' (or Boolean) percolation, in which a Poissonian ensemble of unit balls is placed in $\mathbb{R}^d$.
Each one of the balls can be thought as being an obstacle.
Letting $\mathcal{V}$ stand for the complement of this random set of obstacles (which is sometimes called ``vacant set'' or ``carpet'', see \cite{Szn09} and \cite{NW11}), the primary question one can ask is whether $\mathcal{V}$ contains or not an unbounded connected component with positive probability.
If so, one says that the vacant set $\mathcal{V}$ percolates.
Due to the uniformly boundedness of the obstacles, a number of techniques developed in the study of Bernoulli site percolation can be adapted to this continuum case, see for instance \cite{Gri99}, Section~12.10.

However, for other models containing large obstacles, the induced random environment may feature long-range dependencies, often leading to some intriguing behavior and challenging problems such as in \cite{Hall85}.
We now describe some instances of such models.

K.~Symanzik in his seminal work \cite{Szy69}, introduced a representation of the $\phi^4$ quantum field as a classical gas of Brownian paths which interact when they cross.
This development naturally led to the ideas of loop measures, whose geometry have been intensively investigated, both for planar Brownian motion in relation with SLE processes in \cite{LW04} and for simple random walks in \cite{LF07}.
See also \cite{NW11} and the excellent study in \cite{LeJan08}.
In three dimensions the current knowledge of those models is more restricted, except for the work \cite{RW11}, concerned with the percolative properties of the Brownian loop soup in $\mathbb{R}^3$.


Recently other interesting models took a central stage in the field of random media.
Notably, {\it the random interlacements} on $\mathbb{Z}^d$, $d \ge 3$, introduced by A.-S. Sznitman in \cite{Szn09}.
For this model, the set of obstacles consist of a Poissonian cloud of bi-infinite random walk trajectories modulo time shift.
An intensity parameter $u \geq 0$ controls the amount of trajectories to be removed from the ambient space $\mathbb{Z}^d$ and the complement of these trajectories (the so-called vacant set of random interlacements) was extensively studied in \cite{SS09, SS10, T10, Tei09b}.
It was shown in \cite{Szn09} and \cite{SS09} that the connectivity of the vacant set undergoes a non-trivial phase transition as $u$ crosses a critical threshold.

Another percolation model having similar features is the so-called coordinate percolation, introduced by the second author.
In this model each discrete line parallel to one of the coordinate axis of $\mathbb{Z}^3$ is independently removed with a positive probability $q = 1-p$, and retained with probability $0< p <1$. 
In \cite{Hil11} it was shown that the vacant set left after the removal of lines undergoes a non-trivial phase transition as $p$ varies.
It is important to stress that this model has polynomial decay of connectivity in the super-critical and upper sub-critical phases, see Remark~\ref{r:diff} below.
Polynomial decay and several other properties of coordinate percolation  in $\mathbb{Z}^3$ are remarkably similar to that of P. Winkler's percolation models, see \cite{Gacs00}, including the question of the compatibility of binary sequences and scheduling of random walks (also known as the clairvoyant daemon problem). 

In this article we study a
model governed by a Poisson point process, on the space $\mathbb{L}$ of lines in $\mathbb{R}^d$, having intensity measure $u \mu$.
Here $u$ is a positive real parameter and $\mu$ is, up to a multiplicative constant, the unique Haar measure in $\mathbb{L}$ which is invariant with respect to isometries of $\mathbb{R}^d$, see \eqref{e:mu} for details.

%

Having specified the intensity measure $u \mu$ (see \eqref{e:mu}), a corresponding Poisson point process can be easily constructed in an appropriate probability space $(\Omega, \mathcal{A}, \mathbb{P}_u)$, as we describe in Section~\ref{s:notation}. 
Each element $\omega \in \Omega$ is a point measure, i.e.
\begin{equation}
\omega = {\sum_{i \geq 0}} \delta_{l_i}, \text{ where $l_i$ runs over a countable collection of lines in $\mathbb{R}^d$}.
\end{equation}
We are mainly interested in the set
\begin{equation}
\label{e:Lomega}
\mathcal{L}(\omega) = {\textstyle \bigcup \limits_{l \in \text{supp}(\omega)}} C(l),
\end{equation}
where $C(l)$ stands for the cylinder of radius one around $l$. As well as its complement
\begin{equation}
\label{e:Vomega}
\mathcal{V}(\omega) = \mathbb{R}^d \setminus \mathcal{L}(\omega),
\end{equation}
the so-called `vacant set'.
Intuitively speaking, the set $\mathcal{V}$ represents what is left after we drill through all the lines in the support of $\omega$. 
As for the parameter $u$, it controls the amount of cylinders to be removed from $\mathbb{R}^3$: as $u$ increases, more and more cylinders are drilled, making it increasingly harder for $\mathcal{V}$ to be well connected.

The main contribution of this paper is to prove the following
\begin{theorem}
\label{e:main}
\textnormal{($d=3$)} For $u$ small enough, the vacant set $\mathcal{V}$ contains almost surely an unbounded connected component.
\end{theorem}
See Theorem~\ref{t:main} below for a stronger version of this statement. If one defines the critical parameter by
\begin{equation}
\label{e:ustar}
u_* = \inf\{u \geq 0; \; \mathbb{P}_u [ \text{$\mathcal{V}$ has an unbounded connected component} ] = 0 \},
\end{equation}
then Theorem~\ref{e:main} proves that $u_*$ is strictly positive.

The model was introduced by I.~Benjamini and first studied by J.~Tykesson and D.~Windisch in \cite{TW10b}, where among other results they established the existence of a phase transition for the vacant set left by these cylinders in $\mathbb{R}^d$ when $d \geq 4$.
More specifically in Theorems~4.1 and 5.1 of \cite{TW10b}, they proved that
\begin{gather}
u_* < \infty, \text{ for every $d \geq 3$ and}\\
\label{e:dgeq4}
u_* > 0, \text{ for every $d \geq 4$.}
\end{gather}

The most challenging and physically relevant question concerns to the three-di\-men\-sional case, for which the existence of a percolative phase was still open.
Our result settles the existence of a non-trivial phase transition as the parameter $u$ crosses the non-degenerate threshold $u_*$: The super-critical or percolative phase is the one for which $u<u_*$ and the sub-critical phase is the one for which $u>u_*$.

One of the difficulties in establishing Theorem~\ref{e:main}, is the slow decay of correlations observed in the set $\mathcal{V}$.
As it was observed in Remark~3.2~(4) of \cite{TW10b}, for any $x,y \in \mathbb{R}^d$ with $|x-y| > 2$,
\begin{equation}
\label{e:depend}
\frac{c_{d,u}}{|x-y|^{d-1}} \leq \text{cov}_u(\mathbf{1}_{x \in \mathcal{V}}, \mathbf{1}_{y \in \mathcal{V}}) \leq \frac{c'_{d,u}}{|x-y|^{d-1}},
\end{equation}
where $c_{d,y}$ and $c'_{d,u}$ are positive constants depending on $u$ and $d$ and $\text{cov}_u$ stands for the covariance under the measure $\mathbb{P}_u$.
From \eqref{e:depend} it is clear that in low dimensions the vacant set $\mathcal{V}$ presents a slower decay of correlations, what makes the problem more challenging.

It is worth noticing that an equation similar to \eqref{e:depend} also holds for the vacant set left by random interlacements, but with the exponent $d-1$ replaced by $d-2$, see Remark~1.6 4) in \cite{Szn09}.
We note that also in the case of interlacements, the low dimensional cases are harder.
Indeed, the existence of a percolative phase for random interlacements was first established for $d \geq 7$, in \cite{Szn09}, but only later this result was extended to $d \geq 3$, see \cite{SS09}.

Another difficulty that appears in the present context is the absence of exponential bounds or domination by Boolean percolation (see Remark~\ref{r:diff}).


In order to state what we perceive as the main difficulty to prove  Theorem~\ref{e:main} and to explain why the three dimensional case is qualitatively different from the others, let us briefly describe how the case $d \geq 4$ was handled in \cite{TW10b}.
In that work, the authors restricted their attention to the intersection between $\mathcal{V}$ and $\mathbb{R}^2$ (naturally embedded in $\mathbb{R}^d$).
A similar procedure was also employed in the context of interlacements percolation in \cite{SS09}.
In Theorem~5.1 of \cite{TW10b}, the authors proved that for $d \geq 4$ and for $u$ small enough there exists $\mathbb{P}_u$-a{.}s{.} an unbounded connected component in $\mathcal{V} \cap \mathbb{R}^2$, yielding \eqref{e:dgeq4}.
However, as they also observed, this strategy is destined to fail in three dimensions, as
\begin{equation}
\label{e:noplane}
\begin{array}{c}
\text{for $d = 3$, for every $u > 0$, the set $\mathcal{V}\cap \mathbb{R}^2$ contains}\\
\text{$\mathbb{P}_u$-a.s{.} no unbounded connected component.}
\end{array}
\end{equation}
see Proposition~5.6 of \cite{TW10b}.

In view of \eqref{e:noplane}, in order to establish Theorem~\ref{e:main} we have to search for connections outside the plane $\mathbb{R}^2$.
But, first of all, why would someone be interested in restricting the set $\mathcal{V}$ to $\mathbb{R}^2$?
This is done in order to use the so-called `path duality' of the plane, which roughly speaking, states that
\begin{equation}
\label{e:duality}
\begin{array}{c}
\text{if the connected component of $\mathcal{V}\cap\mathbb{R}^2$ containing the origin is bounded,}\\
\text{then there exists a circuit surrounding the origin in $\mathcal{L} \cap \mathbb{R}^2$,}
\end{array}
\end{equation}
see (5.22) of \cite{TW10b}.
The above statement reduces the task of proving percolation to showing that typical paths in $\mathcal{L} \cap \mathbb{R}^2$ are small.
In our case, we will make use of a statement similar to \eqref{e:duality}, see \eqref{e:Percsigma}.
However, instead of $\mathbb{R}^2$, we will intersect $\mathcal{V}$ with a periodic surface $H$ defined in \eqref{e:H}, see also Figure~\ref{f:setH}.
This surface is contained in the slab $\mathbb{R}^2 \times [0,1000]$ and has two important properties.
First, $H$ is homeomorphic to $\mathbb{R}^2$, which allows us to use duality on $H$ in an indirect way.
Moreover, $H$ is `rough', meaning that its intersection with any fixed cylinder gives rise to small connected components only, see \eqref{e:smallpieces}.

It is striking that there is never percolation on $\mathcal{V} \cap \mathbb{R}^2$, but it is even more surprising that, at the same time, $\mathcal{V} \cap H$ does percolate, as shown in Theorem~\ref{t:main}.
This contrast between the behavior of the random sets $\mathcal{V} \cap \mathbb{R}^2$ and $\mathcal{V}\cap H$ (both satisfying \eqref{e:depend}) is further discussed in Remark~\ref{r:compare}, raising the following question: What property of a given Poissonian cloud of obstacles prevents the existence of a percolative regime?
We hope that this work will bring attention to this question.

Let us briefly explain why is it that $\mathcal{V} \cap \mathbb{R}^2$ never percolates.
In \cite{TW10b}, the authors show that no matter how small $u$ is taken, there are infinitely many triangles (contained in the union of exactly three cylinders in $\mathcal{L}$) that surround the origin in $\mathcal{L} \cap \mathbb{R}^2$.
This is intuitive, since for small values of $u$ we don't expect that several cylinders could cooperate in creating a long dual path.
Therefore, the only way to prevent percolation on $\mathcal{V} \cap \mathbb{R}^2$ is indeed to have few cylinders that alone manage to create a long dual circuit around the origin.
This is certainly possible in $\mathbb{R}^2$, but not in the surface $H$, due to its roughness, see \eqref{e:smallpieces}.
The renormalization scheme developed in Section~\ref{s:renorm} allows us to formalize this heuristic argument, providing a way to isolate the collective and individual influence of obstacles.

Next we briefly explain the novelties on the renormalization technique presented here.
We first define a rapidly increasing sequence of scale lengths $(a_n)_{n \geq 0}$, see \eqref{eq:scal}. 
Our aim is to analyze the probability $p_n$ that
\begin{equation}
\label{e:pn_describe}
\begin{array}{c}
\text{there exists some path in } \mathcal{L}\cap H \text{ connecting the ball of radius } \\
a_n/10\text{ to the surface of the ball of radius } a_n \text{ around the origin.}
\end{array}
\end{equation}
It can be easily seen that the event in \eqref{e:pn_describe} implies the occurrence of similar events in two smaller balls (of radius $a_{n-1}$) which are far apart, see \eqref{e:conn_inc}.
It is therefore tempting to bound $p_n$ in terms of $p_{n-1}^2$, but for this we need an approximate independence between what happens to $\mathcal{V}$ inside these two smaller balls.
In \cite{TW10b}, the authors accomplish this by plugging in a bound on this dependence which resembles \eqref{e:depend}.
This is enough to establish the result for $d \geq 4$, but for $d = 3$, the problem is fundamentally more complicated, as we can infer from \eqref{e:noplane}.

At this point, we introduce an auxiliary sequence $q_n$ that corresponds to the probability of the events appearing in \eqref{e:pn_describe} with two cylinders being deterministically added to the random set $\mathcal{L}(\omega)$.
A delicate balance between the probability that two distant balls intersect the same cylinders and a combinatorial factor for the possible choices of these two balls makes it possible to construct a contracting recursion relation between $(p_n, q_n)$ and $(p_{n-1}, q_{n-1})$.
Finally, we use the roughness of $H$ to trigger these recursion relations, i.e. show that $p_0$ and $q_0$ are small if $u$ is small, finishing the proof of Theorem~\ref{e:main}.

This paper is organized as follows:
In Section~\ref{s:notation} we give a rigorous construction of the model and introduce the notation used throughout the text.
In Section~\ref{s:renorm} we state Theorem~\ref{t:main} which is our main result and introduce the mathematical setting for the renormalization used in its proof, finishing with recurrence relations between scales.
Section~\ref{s:trigger} is dedicated to triggering the recurrence relations obtained previously.
Finally, in Section~\ref{s:main} we join the results of the two previous sections in order to prove Theorem~\ref{t:main}.
We also include an Appendix, where we prove some basic geometric facts that are useful in the proof of the recursion relations.

\section{Notation}
\label{s:notation}

Throughout the text $c$ or $c'$ denote strictly positive constants, with value changing from place to place.
Dependence of constants on additional parameters appears in the notation.
For instance $c_u$ denotes a positive constant possibly depending on $u$.
Numbered constants, such as $c_0, c_1, \dots$ are fixed according to their first appearance in the text.


As we have mentioned in the last section, we let
\begin{equation}
\label{e:L}
\mathbb{L} \text{ denote the space of all $1$-dimensional affine subspaces of $\mathbb{R}^3$}.
\end{equation}

We introduce a measure $\mu$ in the space $\mathbb{L}$ of lines in $\mathbb{R}^3$ following the construction in \cite{TW10b}.
For this, let $e_i$, $i=1,2,3$, stand for the vectors of the canonical orthonormal basis of $\mathbb{R}^3$ and define $l$ to be the axis $\{t \cdot e_3; \; t \in \mathbb{R}\}$.
We also let $\mathbb{R}^2$ correspond in the natural way to the plane $\{(x,y,0); \; x, y \in \mathbb{R}\}$, orthogonal to $l$, endowed with the Lebesgue measure $\lambda$.
Consider also the group $SO_3$ of rigid rotations in $\mathbb{R}^3$ endowed with the natural topology and the unique Haar measure $\nu$ normalized in a way that $\nu(SO_3) = 1$.
Then we define
\begin{equation}
\begin{split}
\alpha: \mathbb{R}^2 \times SO_3 & \to \mathbb{L}\\
(x, \theta) & \mapsto \theta(\tau_x(l)),
\end{split}
\end{equation}
where $\tau_x$ is the translation map from $\mathbb{R}^3$ onto itself defined by $y \mapsto x + y$.

With this definition, we can endow the set $\mathbb{L}$ with the finest topology that makes the map $\alpha$ continuous.
Let $\mathcal{B}(\mathbb{L})$ stand for the corresponding Borel $\sigma$-algebra.
We can thus introduce the measure $\mu$ on $(\mathbb{L}, \mathcal{B}(\mathbb{L}))$:
\begin{equation}
\label{e:mu}
\mu = \alpha(\lambda \otimes \nu).
\end{equation}
We note that $\mu$ is (up to multiplicative constants) the unique Haar measure on $\mathbb{L}$ which is invariant under isometries of $\mathbb{R}^3$.

We now consider the space of point measures
\begin{equation}
\begin{split}
\Omega  = \Big\{ \omega = \sum_{i \geq 0} \delta_{l_i}; \; & l_i \in \mathbb{L} \text{ and }
\omega(A) < \infty, \text{ for every compact $A \in \mathcal{B}(\mathbb{L})$} \Big\},
\end{split}
\end{equation}
endowed with the $\sigma$-algebra $\mathcal{A}$ generated by the evaluation maps $\phi_A: \omega \mapsto \omega(A)$, for $A \in \mathcal{B}(\mathbb{L})$.

We are now in the position to define the main process we intend to analyze.
For this, fix some $u \geq 0$ and define the probability space $(\Omega, \mathcal{A}, \mathbb{P}_u)$ of a Poisson point process with intensity measure given by $u \cdot \mu$.
The expectation operator associated with $\mathbb{P}_u$ will be denoted by $\mathbb{E}_u$.
For a reference for this construction, see for instance Proposition~3.6 of \cite{R08}.

Due to the fact that $\mu$ is invariant under the isometries of $\mathbb{R}^3$, one can show that the law $\mathbb{P}_u$ governing this Poisson point process is also invariant under such transformations (see Remark~2.1 of \cite{TW10b}).
Furthermore the law $\mathbb{P}_u$ can be shown to be ergodic under translations in the sense that will be described in Section \ref{s:main}.

The Euclidean distance in $\mathbb{R}^3$ or in $\mathbb{R}^2$ will be denoted by dist$(\cdot,\cdot)$.
For a point $x \in \mathbb{R}^3$ and $r>0$ we denote $B(x,r) = \{y \in \mathbb{R}^3; \text{ dist}(x,y)\leq r\}$ and for a set $A \subset \mathbb{R}^3$ we denote $B(A,r) = \cup_{x \in A}B(x,r)$.
For a line $l \in \mathbb{L}$ let $C(l) = B(l,1)$ be the cylinder of radius one and axis equal to $l$.
We denote by $\mathbb{C}$ the set of all cylinders of radius one: $\{C(l);~ l\in \mathbb{L}\}$.

We let $\mathcal{L}( \omega)$ be the `thickening' of the lines in the support of $\omega \in \Omega$, i.e.
\begin{equation}
\mathcal{L}(\omega) = \bigcup_{l \in \supp(\omega)} C(l),
\end{equation}
as well as its complement
\begin{equation}
\mathcal{V}(\omega) = \mathbb{R}^3 \setminus \mathcal{L}(\omega),
\end{equation}
also referred to as the `vacant set left by the cylinders'.

As proved in Proposition~5.6 of \cite{TW10b},
\begin{equation}
\label{e:noplane2}
\begin{array}{c}
\text{in $d = 3$, for every plane $K \subset \mathbb{R}^3$ and every $u > 0$, there is no}\\
\text{percolation in $\mathcal{V} \cap K$, almost surely with respect to $\mathbb{P}_u$.}
\end{array}
\end{equation}
This means that in order to establish the existence of an unbounded component in $\mathcal{V}$ we need to search for components that may exit planes.
As it turns out, it is enough to consider the vacant set $\mathcal{V}$ intersected with the slab $\mathbb{R}^2 \times [0,1000]$.
The number $1000$ carries no special significance and it was chosen large enough so that the proof of Proposition \ref{prop:pq_tozero} could be carried out.

\begin{remark}
\label{r:diff}
1) As it was established in Remark~3.2 1) and 3) in \cite{TW10b}, the model considered in this article does not dominate nor is dominated by any (constant radius) Boolean percolation model, indicating that the techniques currently available for the Boolean and the Bernoulli percolation may not work to establish results in the current context. 
This is well illustrated in \cite{TW10b}, Remark~3.2 2), where the authors rule out the so-called exponential bounds which are very useful for Boolean and Bernoulli percolation.

2) After establishing the existence of a non-trivial phase transition, one could be interested in studying the uniqueness of such transition.
Roughly speaking this corresponds to study whether the correlation length undergoes any abrupt change, besides the one expected at $u_*$.
Both Boolean percolation and Bernoulli percolation present a unique phase transition in this sense.
Moreover, for these models the two points function decays exponentially both in the sub-critical phase (Menshikov's theorem) and in the finite clusters of the super-critical phase (see Theorem~(8.18) in \cite{Gri99}, p.205).

It is important to notice that these classical results may fail in the presence of long-range dependence.
For the coordinate percolation one can show that the two points function decays slower than a polynomial in the super-critical phase, see \cite{Hil11}.
On the other hand in the sub-critical phase for low enough parameter $p$, this rate is exponential.
Is is still not known whether the decay is exponential throughout all the sub-critical phase.
For the P{.} Winkler percolation process, it has been shown in \cite{Gacs00} that the decay is also polynomial throughout all the super-critical phase.
For interlacements percolation, this decay is known to be no faster than a stretched exponential, see Theorem~3.6 of \cite{T10}.
It is an interesting problem to study the above questions for the cylinder's percolation model, see remark \eqref{r:open} 3).
\end{remark}

\section{The renormalization scheme}
\label{s:renorm}

In this section we start to develop the renormalization scheme that leads to the proof that $\mathcal{V}$ percolates within the slab $\mathbb{R}^2 \times [0,1000]$ provided that the parameter $u$ is small enough.
As we have mentioned above, we will define a surface contained in this slab.
For this end we start by defining a hexagonal tilling of the plane $\mathbb{R}^2 \subset \mathbb{R}^3$.

First consider the set
\begin{equation}
G = \big\{2000 (n + m \, e^{i\tfrac{\pi}3})  ; \; \text{for } n, m \in \mathbb{Z}\big\},
\end{equation}
which will correspond to the centers of the faces defining the tilling.
The boundary $\mathcal{H}$ of the hexagonal tilling is defined as the set of points $x \in \mathbb{R}^2$ such that the distance between $x$ and  $G$ is attained for more than one point in $G$.
We also denote by $\hex \subset \mathbb{R}^2$ the face of the tilling containing the origin, or more precisely, $\hex$ is the closure of the connected component of $\mathbb{R}^2 \setminus \mathcal{H}$ containing the origin.

Consider the map $\dist(\cdot,\mathcal{H}) : \mathbb{R}^2 \to \mathbb{R}_+$, which associates to $x \in \mathbb{R}^2 \subset \mathbb{R}^3$ the distance $\dist(x,\mathcal{H})$ between $x$ and $\mathcal{H}$. 
We will be interested in the graph of this map regarded as a subset of $\mathbb{R}^3$, i.e.
\begin{equation}
\label{e:H}
H = \big\{ (x,t) \in \mathbb{R}^2 \times \mathbb{R}; \; x \in \mathbb{R}^2 \text{ and } t = \dist(x,\mathcal{H})\big\}, \text{ see Figure~\ref{f:setH}.}
\end{equation}
Note that $H$ is a surface contained in the slab $\mathbb{R}^2 \times [0,1000]$ mentioned above and that $(0,0,1000)$ belongs to $H$.

   \begin{center}
   \begin{figure}
   \includegraphics[width = 90mm]{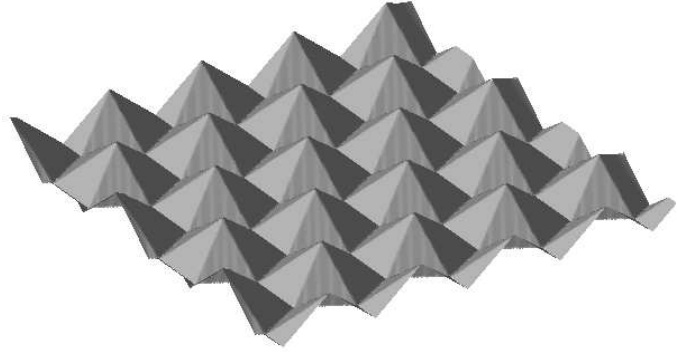}
   \caption{A piece of the set $H$.} \label{f:setH}
   \end{figure}
   \end{center}

We now state the main result in this article
\begin{theorem}
\label{t:main}
For $d = 3$, for $u$ small enough,
\begin{equation}
\mathbb{P}_u [\mathcal{V} \cap H \text{ has an unbounded connected component}] = 1.
\end{equation}
In particular, $u_* > 0$ in three dimensions.
\end{theorem}

We denote by $\pi$ the orthogonal projection from $\mathbb{R}^3$ onto $\mathbb{R}^2$.
When restricted to $H$, $\pi$ defines a homeomorphism between $H$ and $\mathbb{R}^2$.

\begin{remark}
\label{r:compare}
Let us briefly compare \eqref{e:noplane} to Theorem~\ref{t:main}. Note that both $\mathcal{V} \cap \mathbb{R}^2$ and $\pi(\mathcal{V} \cap H)$ are random subsets of $\mathbb{R}^2$ which are ergodic (see Lemma~3.3 in \cite{TW10b} and Section~\ref{s:main}). Moreover, they present similar decays of correlation, indeed they both satisfy \eqref{e:depend}.
However, \eqref{e:noplane} shows that $\mathcal{V} \cap \mathbb{R}^2$ does not present a phase transition, while Theorem~\ref{t:main} proves that $\pi(\mathcal{V} \cap H)$ does.
In this paper, we explain the discrepancy between these two processes in Proposition~\ref{prop:pq_tozero}, which holds true only for $\pi(\mathcal{V} \cap H)$ (see Remark~\ref{r:diffHR2}).
This raises the following question: for which models of Poissonian obstacles can one establish the existence (or absence) of a non-degenerate phase transition as the intensity $u$ varies?
\end{remark}

The proof of this theorem uses in an indirect way the duality of $\mathbb{R}^2$.
More precisely,
\begin{equation}
\label{e:dual}
\begin{array}{c}
\text{if the connected component of $\mathcal{V}(\omega) \cap H$ containing $(0,0,1000)$ is bounded,}\\
\text{then there exists a circuit in $\pi(\mathcal{L}(\omega) \cap H)$ surrounding the origin in $\mathbb{R}^2$,}
\end{array}
\end{equation}
see the paragraph before equation \eqref{e:Percsigma} for a proof.

In view of \eqref{e:dual}, we should pursue a bound on the existence of large paths in $\pi(\mathcal{L}(\omega) \cap H)$. For this, we follow a renormalization argument inspired in \cite{TW10b} and \cite{Szn09}. But first we introduce some notation. Let
\begin{equation}
\label{e:gamma}
a_0 \in [288^6, \infty) \quad \text{and} \quad \gamma = 7/6.
\end{equation}
The choice of the parameter $a_0$ will be made latter, but it is important to notice that all the statements we make in this section (and in the Appendix) hold true for any $a_0$ as above.
Moreover, in accordance to our convention, all the constants appearing in this section are independent of the specific choice of $a_0$ unless stated otherwise.

Let us also consider the following sequence of scales:
\begin{equation}
\label{eq:scal}
a_n = a_0^{\gamma^n}.
\end{equation}
Note that this sequence grows faster than exponentially. In fact
\begin{equation}
\label{e:an_ratio}
\Big(\frac{a_n}{a_{n-1}} \Big) = a_{n-1}^{\gamma-1}.
\end{equation}
The reason to impose that $a_0 \geq 288^6$, is to guarantee that
\begin{equation}
\label{e:a_lower0}
a_n \geq 8000 \quad \text{and} \quad a_{n+1} \geq 288 \, a_{n}, \text{ for every $n \geq 0$},
\end{equation}
which will be useful for instance in the proofs of Lemmas~\ref{l:cove_2}, \ref{l:p_rec} and \ref{l:cove_3} below.

For $x \in \mathbb{R}^2$ and $r>0$ set $S(x,r) = \{y \in \mathbb{R}^2; \text{ dist}(x,y) \leq r\}$ and $\partial S(x,r)$ the boundary of $S(x,r)$ in $\mathbb{R}^2$.
For $n \geq 0$ and any $x \in \mathbb{R}$, we define the following function $A_n: \mathbb{R}^2 \times \Omega \to \{0,1\}$
\begin{equation}
\label{e:Anx}
A_n (x, \omega) = \mathbf{1}{\{S(x, a_n/10) \leftrightarrow \partial{S(x,a_n)} \text{ in } \pi(\mathcal{L}(\omega) \cap H)\}},
\end{equation}
where the event appearing in the right-hand side of the previous equation is: `there exists a continuous path starting at a point in $S(x, a_n/10)$ and ending at a point in $\partial{S(x, a_n)}$ and having its image contained in $\pi(\mathcal{L}(\omega) \cap H)$'.
Note that for a fixed $x \in \mathbb{R}^2$ and a point measure $\omega' \in \Omega$ with finite support,
\begin{equation}
\text{the function } \omega \mapsto A_n(x,\omega + \omega') \text{ is measurable.}
\end{equation}
To see this, first observe that the set $\mathcal{L}(\omega + \omega') \cap H \cap (S(x,a_n) \times \mathbb{R})$ is given by the union of finitely many convex and compact sets, as it can be seen by splitting the set $H$ into its faces.
The rest of the proof follows the same arguments as for Lemma~5.2 in \cite{TW10b}.

Denoting by $A_n(x)$ the random variable $A_n(x, \omega): \Omega \to \{0,1\}$, we can define
\begin{equation}
\label{e:pn}
p_n (u) = \sup_{x \in \mathbb{R}^2} \mathbb{E}_u [A_n(x)] = \sup_{x \in \hex} \mathbb{E}_u [A_n(x)],
\end{equation}
where for the second equality we used the periodicity of the set $H$ and the translation invariance of $\mathbb{P}_u$.
In order to prove Theorem~\ref{t:main}, we first need to show that for $u$ small enough the sequence $p_n(u)$ decays fast with $n$.
This will be obtained via a recursion relation that we develop below.

We define for each $n \geq 1$, the lattice
\begin{equation}
\label{e:Jn}
\mathcal{J}_n = \Big( \frac{a_{n-1}}{10} \Big) \cdot \mathbb{Z}^2.
\end{equation}
and for $i = 1,2,3,4$
\begin{equation}
\label{e:Hin}
\mathcal{H}_n^i = \Big\{ x \in \mathcal{J}_n ; S(x, a_{n-1}) \cap {\textstyle \bigcup \limits_{x_0 \in \hex}} \partial{S} \big( x_0, \tfrac{(i+1)a_n}6 \big) \not = \varnothing \Big\}.
\end{equation}
The reason to consider four different spheres ($i = 1, \dots, 4$) is explained in the paragraph before Lemma~\ref{l:cove_3}.

   \begin{center}
    \begin{figure}
    \includegraphics[width = 110mm]{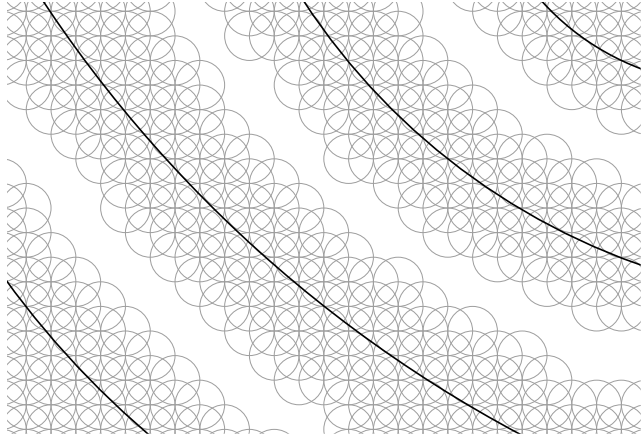}
    \caption{For $i = 1, \dots, 4$, the figure ilustrates a section of the sets $\bigcup_{x_0 \in \hex} \partial{S} \big( x_0, \tfrac{(i+1)a_n}6 \big)$ (in black) and the balls $S(x,a_{n-1}/10)$, for $x \in \mathcal{H}^i_n$ (in light gray).} \label{f:Hin}
    \end{figure}
    \end{center}


The sets $\mathcal{H}^i_n$ defined in (\ref{e:Hin}) satisfy three important properties that will be useful for proving Theorem~\ref{t:main}. They are stated in Lemmas~\ref{l:cove_1}, \ref{l:cove_2} and \ref{l:cove_3} and, even though their proofs are quite simple, we include them in the Appendix for the convenience of the reader.

The first of these properties states that the sets $\mathcal{H}^i_n$ can be used to define coverings of the spheres $\partial{S} ( x_0, {(i+1)a_n}/6 )$, see Figure~\ref{f:Hin}. More precisely,

\begin{lemma}
\label{l:cove_1}
For $(a_n)_{n \geq 0}$ as in \eqref{e:new_scales}, if we let $\mathcal{H}_n^i$ be defined as in (\ref{e:Hin}), then
\begin{equation}
\label{eq:cove_1}
\bigcup_{x_0 \in \hex} \partial S \left( x_0, \tfrac{(i+1)a_n}{6} \right) \subset \bigcup_{x \in \mathcal{H}_n^i} S \left( x, \tfrac{a_{n-1}}{10} \right),
\end{equation}
for all $i = 1, \ldots, 4$, and $n \geq 1$.
\end{lemma}

\begin{proof}The proof of this lemma is postponed to the Appendix. \end{proof}

In Lemma~\ref{l:pbound} below, we are going to use union bounds on $x \in \mathcal{H}^i_n$, therefore we need the following control on the cardinality of these sets.

\begin{lemma}
\label{l:cove_2}
There exists a positive constant \newconstant{entrop} $\useconstant{entrop}$ such that, for any $(a_n)_{n \geq 0}$ as in \eqref{e:new_scales} if we let $\mathcal{H}_n^i$ be defined as in (\ref{e:Hin}), then
\begin{equation}
\label{eq:cove_2}
\max_{i = 1,\ldots, 4} \left| \mathcal{H}_n^i \right| \leq \useconstant{entrop} \big( \tfrac{a_n}{a_{n-1}}\big),
\end{equation}
for all $n \geq 1$. Note that, in accordance with our convention on constants, $\useconstant{entrop}$ does not depend on the specific choice of the scale parameter $a_0$.
\end{lemma}

\begin{proof}The proof of this lemma is also presented in the Appendix. \end{proof}

As mentioned above, in order to prove that for some $u$ small enough the probability $p_n(u)$ decays with $n$, we are going to obtain a recursion relation between $p_n(u)$ and $p_{n-1}(u)$.
The next lemma gives an indication why this should be possible, as it relates $p_n$ with the random variable $A_{n-1}(\cdot)$.

\begin{lemma}
\label{l:pbound} Fix $u > 0$ and recall the definitions of $p_n(u)$ in \eqref{e:pn} and $\mathcal{H}^i_n$ in \eqref{e:Hin}. For any pair of distinct $i_1, i_2 \in \{1,2,3,4\}$, it holds that
\begin{equation}
\label{e:pbound}
p_n(u) \leq {\useconstant{entrop}}^2 \Big(\frac{a_n}{a_{n-1}} \Big)^2 \sup_{{x_1 \in \mathcal{H}^{i_1}_n}, \;  {x_2 \in \mathcal{H}^{i_2}_n}} \mathbb{E}_u \left[ A_{n-1} (x_1) A_{n-1} (x_2) \right],
\end{equation}
for all $n \geq 1$, where the constant $\useconstant{entrop} > 0$ is the one appearing in Lemma~\ref{l:cove_2}.
\end{lemma}

\begin{proof}
Fix an $x_0 \in \hex$ as in the right-hand side of \eqref{e:pn}.
By the property of the sets $\mathcal{H}^i_n$ stated in Lemma~\ref{l:cove_1}, we have that for any $j = 1,2$ the family $\{S(x,a_{n-1}/10)\}_{x \in \mathcal{H}^{i_j}_n}$ covers the sphere $\partial{S}(x_0, (i_j+1)a_n/6)$. Therefore, any path connecting $S(x_0, a_n/10)$ to $\partial{S}(x_0, a_n)$ must intersect a ball $S\big(x_j,(a_{n-1})/10\big)$ with $x_j \in \mathcal{H}^{i_j}_n$ for both $j = 1,2$.
It also follows that this path must connect $S\big(x_j,(a_{n-1})/10\big)$ to $\partial{S}(x_j, a_{n-1})$, for $j=1,2$.
In particular we have the following inclusion
\begin{equation}
\label{e:conn_inc}
\begin{split}
\big\{& S(x, \tfrac{a_n}{10})  \leftrightarrow \partial{S(x,a_n)} \text{ in } \pi(\mathcal{L}(\omega) \cap H)\big\}\\
& \subseteq \bigcup_{x_1 \in \mathcal{H}^{i_1}_n, \; x_2 \in \mathcal{H}^{i_2}_n} \;\;\bigcap_{j = 1,2} \Big\{ S(x_j, \tfrac{a_{n-1}}{10}) \leftrightarrow \partial{S} (x_j, a_{n-1}) \text{ in } \pi(\mathcal{L}(\omega) \cap H) \Big\}
\end{split}
\end{equation}

Using Lemma~\ref{l:cove_2}, we have that the above union has no more than ${\useconstant{entrop}}^2 (a_n/a_{n-1})^2$ members, so that
\begin{equation}
\mathbb{E}_u[A_n(x_0)] \leq \useconstant{entrop}^2 \big(\tfrac{a_n}{a_{n-1}}\big)^2 \sup_{{x_1 \in \mathcal{H}^{i_1}_n}, \;  {x_2 \in \mathcal{H}^{i_2}_n}} \mathbb{E}_u \left[ A_{n-1} (x_1) A_{n-1} (x_2) \right].
\end{equation}
The result now follows by taking the supremum over $x_0 \in \hex$.
\end{proof}

We now have to deal with the dependence between the two indicator functions appearing in the right-hand side of \eqref{e:pbound}.
Let us mention here that the technique employed to bound this dependence in Theorem~5.1 of \cite{TW10b} is destined to fail, see Remark~\ref{r:track} below and Proposition~5.6 of \cite{TW10b}.

Therefore, in order to keep track of more refined details on the dependence between $A_{n-1} (x_1)$ and $A_{n-1} (x_2)$, we introduce the following sequence
\begin{equation}
\label{e:qn}
q_n (u) = \sup_{x \in \hex} \;\; \sup_{l_1, l_2 \in \mathbb{L}} \mathbb{E}_u [A_n(x, \omega + \delta_{l_1} + \delta_{l_2})],
\end{equation}
where the above expectation is taken with respect to $\omega \in \Omega$.
Note that the random variable $A_n(x, \omega + \delta_{l_1} + \delta_{l_2})$ corresponds to $A_n(x,\omega)$ after we add two deterministically positioned lines to the random point measure $\omega$.
The quantity $q_{n-1}(u)$ will help us to control the dependence between the random variables $A_{n-1} (x_1)$ and $A_{n-1} (x_2)$ for $x_1 \in \mathcal{H}^{i_1}_n$ and $x_2 \in \mathcal{H}^{i_2}_n$.
This control is attained by considering the number of cylinders intersecting at the same time some suitable neighborhoods of $x_1$ and $x_2$ in three different scenarios:
The first in which this number is equal to zero, the second in which this number is either one or two and the third in which this number is at least three.
In the first scenario we can consider $A_{n-1}(x_1)$ and $A_{n-1}(x_2)$ as being independent.
The probability that the third scenario occurs is sufficiently small for our purposes.
In the second scenario we dominate the dependencies by adding two cylinders to the process.
That is the point where $q_n(u)$ will be useful and this explains why two lines appear in its definition.
The details are carried out in the lemma below.

Given two sets $A, B \subset \mathbb{R}^3$ that are either open or compact, we define the following subsets of $\mathbb{L}$.
\begin{equation}
\label{e:LA}
L_A = \{l \in \mathbb{L}; C(l) \text{ intersect } A\}
\end{equation}
\begin{equation}
\label{e:LAB}
L_{A,B} = \{l \in \mathbb{L}; C(l) \text{ intersects both $A$ and $B$}\}.
\end{equation}
We refer to the paragraph below Equation (2.11) in \cite{TW10b} for an explanation concerning the measurability of these sets.

\begin{lemma}
\label{l:EA1A2}
Fix $n \geq 1$, two distinct $i_1, i_2 \in \{1,2,3,4\}$ and points $x_1 \in \mathcal{H}^{i_1}_n$ and $x_2 \in \mathcal{H}^{i_2}_n$. Defining $D_j = S(x_j, a_{n-1}) \times [0,1000]$ for $j = 1,2$, we have
\begin{equation*}
\mathbb{E}_u \left[ A_{n-1} (x_1) A_{n-1} (x_2) \right] \leq p_{n-1}^2(u) + \mathbb{P}_u \big[ \omega(L_{D_1, D_2}) \geq 3 \big] + \mathbb{P}_u \big[ 1 \leq \omega(L_{D_1, D_2}) \leq 2 \big] q_{n-1}^2(u).
\end{equation*}
\end{lemma}

Note that the first term in the right-hand side of the above equation corresponds to the natural bound that would be obtained if $A_{n-1} (x_1)$ and $A_{n-1} (x_2)$ were independent. 
Roughly speaking, the other two terms respectively account for the possibilities of `high' and `medium' interaction between $A_{n-1} (x_1)$ and $A_{n-1} (x_2)$.

\begin{proof}
We consider the partition of $\Omega$ into the three disjoint sets given by: $[\omega(L_{D_1, D_2}) = 0]$, $[\omega(L_{D_1, D_2}) \geq 3]$ and $[1 \leq \omega(L_{D_1, D_2}) \leq 2]$.
These three events will respectively correspond to the three terms in the right hand side of the above equation.

Let us first show that $\mathbb{E}_u [ A_{n-1} (x_1) A_{n-1} (x_2);~ \omega(L_{D_1, D_2}) = 0] \leq p_{n-1}^2(u)$. Indeed
\begin{equation}
\begin{split}
\mathbb{E}_u & [A_{n-1} (x_1) A_{n-1}(x_2);~ \omega(L_{D_1,D_2}) = 0 ] \\
& = \mathbb{E}_u \Big[A_{n-1} \big( x_1, \mathbf{1}_{\mathbb{L}\slash L_{D_2}} \cdot \omega \big) A_{n-1} \big( x_2, \mathbf{1}_{L_{D_2}} \cdot \omega \big) ; \omega(L_{D_1,D_2}) = 0 \Big] \\
& \leq \mathbb{E}_u \Big[A_{n-1} \big(x_1, \mathbf{1}_{\mathbb{L}\slash L_{D_2}} \cdot \omega \big) \Big] \mathbb{E}_u
\Big[ A_{n-1} \big( x_2, \mathbf{1}_{L_{D_2}} \cdot \omega \big) \Big] \leq p_{n-1}^2 (u).
\end{split}
\end{equation}
In the above estimate, we first used that when $\omega(L_{S_1,S_2}) = 0$, we have that $A_{n-1}(x_2, \omega) = A_{n-1}(x_2, \mathbf{1}_{L_{D_2}} \cdot \omega )$ and $A_{n-1}(x_1, \omega) = A_{n-1}(x_2, \mathbf{1}_{\mathbb{L} \setminus L_{D_2}} \cdot \omega )$. Then we neglected the intersection with $\omega(L_{D_1,D_2}) = 0$ and used the independence between the random variables $A_{n-1}(x_2, \mathbf{1}_{L_{D_2}} \cdot \omega )$ and $A_{n-1}(x_2, \mathbf{1}_{\mathbb{L} \setminus L_{D_2}} \cdot \omega )$ (note that they depend on the realization of a Poisson point process in disjoint sets).

It is clear that $\mathbb{E}_u [ A_{n-1} (x_1) A_{n-1} (x_2), \omega(L_{D_1, D_2}) \geq 3] \leq \mathbb{P}_u \big[ \omega(L_{D_1, D_2}) \geq 3 \big]$. Therefore, all we need to do in order to finish the proof is to show that
\begin{equation}
\label{e:w12}
\mathbb{E}_u [ A_{n-1} (x_1) A_{n-1} (x_2);~ 1 \leq \omega(L_{D_1, D_2}) \leq 2] \leq \mathbb{P}_u \big[ 1 \leq \omega(L_{D_1, D_2}) \leq 2 \big] q_{n-1}^2(u).
\end{equation}

For this, let us define for each given $\omega' \in \Omega$,  the following function
\begin{equation}
\phi(\omega') = \mathbb{E}_u \Big[A_{n-1} \big(x_1, \mathbf{1}_{L_{D_1} \setminus L_{D_2}} \cdot \omega + \omega' \big) \Big] \mathbb{E}_u \Big[A_{n-1} \big(x_2, \mathbf{1}_{L_{D_2} \setminus L_{D_1}} \cdot \omega + \omega' \big) \Big],
\end{equation}
where again the above expectations are taken with respect to $\omega \in \Omega$. Intuitively speaking, the above function is the product of the expectations of $A_{n-1}$ in $D_1$ and $D_2$ after a fixed penalization  $\omega'$ is introduced into the point measure.

We note that $\mathbf{1}_{L_{D_1} \setminus L_{D_2}}\cdot \omega$, $\mathbf{1}_{L_{D_2} \setminus L_{D_1}}\cdot \omega$ and $\mathbf{1}_{L_{D_1, D_2}}\cdot \omega$ are independent and
\begin{equation}
\begin{split}
& A_{n-1}(x_1, \omega) = A_{n-1} \big(x_1, \mathbf{1}_{L_{D_1} \setminus L_{D_2}}\cdot \omega + \mathbf{1}_{L_{D_1, D_2}}\cdot \omega \big),\\
& A_{n-1}(x_2, \omega) = A_{n-1} \big(x_2, \mathbf{1}_{L_{D_2} \setminus L_{D_1}}\cdot \omega + \mathbf{1}_{L_{D_1, D_2}}\cdot \omega \big).
\end{split}
\end{equation}
Which implies that
\begin{equation}
\label{e:A1A2}
\mathbb{E}_u \big[ A_{n-1} (x_1) A_{n-1} (x_2) \big| \mathbf{1}_{L_{D_1,D_2}} \cdot \omega \big] = \phi\big(\mathbf{1}_{L_{D_1,D_2}} \cdot \omega\big),
\end{equation}
almost surely.

Note also that
\begin{equation}
\begin{split}
& \sup \big\{ \phi(\omega'); \; \omega' \in \Omega \text{ satisfying } \text{supp}(\omega') \subseteq L_{D_1, D_2} \text{ and } 1 \leq \omega'(L_{D_1, D_2}) \leq 2 \big\}\\
& \leq \sup_{l_1, l_2 \in L_{D_1,D_2}} \; \mathbb{E}_u[A_{n-1}(x_1, \omega+\delta_{l_1} + \delta_{l_2})]\mathbb{E}_u[A_{n-1}(x_2, \omega+\delta_{l_1} + \delta_{l_2})] \leq q^2_{n-1}(u).
\end{split}
\end{equation}
This, together with \eqref{e:A1A2} implies \eqref{e:w12} and finishes the proof of the lemma.
\end{proof}

\begin{remark}
\label{r:track}
Let us briefly mention here how Lemma~\ref{l:EA1A2} improves the technique employed in Theorem~5.1 of \cite{TW10b} to bound the dependence between $A_{n-1} (x_1)$ and $A_{n-1} (x_2)$. 
Roughly speaking, in (5.16) of \cite{TW10b}, they obtained that
\begin{equation*}
\mathbb{E}_u \left[ A_{n-1} (x_1) A_{n-1} (x_2) \right] \leq p_{n-1}^2(u) + \mathbb{P}_u \big[ \omega(L_{D_1, D_2}) \geq 1 \big].
\end{equation*}
Here, by considering the case $1 \leq \omega(L_{D_1, D_2}) \leq 2$ separately, we can obtain better exponents for our induction relations (see Lemma~\ref{l:p_rec}) allowing this technique to work in the case $d = 3$.
\end{remark}

From Lemma~\ref{l:EA1A2}, it is clear that we will need to bound probabilities of the form $\mathbb{P}_u \big[ \omega(L_{D_1, D_2}) \geq k \big]$, for $k \geq 1$. This is done with the help of the following

\begin{lemma}
\label{l:muL}
\newconstant{muL}
For fixed $x_1, x_2 \in \mathbb{R}^2$ and $s \geq 1$, define the sets $D_1 = S (x_1, s) \times [0,1000]$ and $D_2 = S (x_2, s) \times [0,1000]$, c.f. Lemma~\ref{l:EA1A2}. Then
\begin{equation}
\mu(L_{D_1, D_2}) \leq \useconstant{muL} \frac{s^2}{r^{2}},
\end{equation}
where $r = \dist \{S(x_1,s), S(x_2,s)\}$.
\end{lemma}

\begin{proof}
In case $r<4$, the result follows easily from Lemma~2.2 of \cite{TW10b} using the fact that $L_{D_1,D_2} \subset L_{D_1}$ and $s \geq 1$.

Supposing that $r \geq 4$, for $i =1,2$, let $R_i = \partial{S}(x_i, s) \times [0,1000]$ and consider a covering $\mathcal{R}_i$ of $R_i$ with no more than $cs$ balls of radius one and centered in a point of $R_i$.
By the convexity of a cylinder $C \in \mathbb{C}$, if $C$ intersects both $D_1$ and $D_2$, then it ought to intersect $R_1$ and $R_2$.
Therefore it touches at least two balls $B_1 \in \mathcal{R}_1$ and $B_2 \in \mathcal{R}_2$, and the centers of $B_1$ and $B_2$ are within distances at least $4$.
Using \cite{TW10b}, Lemma 3.1, with $d=3$ we have that
\begin{equation}
\mu (L_{D_1, D_2}) \leq \sum_{\substack{B_1 \in \mathcal{R}_1\\ B_2 \in \mathcal{R}_2}} \mu(L_{B_1, B_2}) \leq \frac{\useconstant{muL} s^2}{r^{2}}
\end{equation}
This finishes the proof of the lemma.
\end{proof}

From now on we fix the parameter $u > 0$ assuming it to be not larger than one. This assumption is made only to simplify the calculations. Since the value of $u$ will be fixed, we omit it in the notations $p_n(u)$ and $q_n(u)$, writing simply $p_n$ and $q_n$.

In Lemmas~\ref{l:p_rec} and \ref{l:q_rec} below, we develop a system of recurrence relation between $p_n$'s and $q_n$'s.

\begin{lemma}
\label{l:p_rec}
\newconstant{prec}
There exists a positive constant $\useconstant{prec}$ such that, for any $(a_n)_{n \geq 0}$ as in \eqref{e:new_scales} and for all $n \geq 1$,
\begin{equation}
\label{e:p_rec}
p_n \leq \useconstant{prec}(a_{n-1}^{\gamma-1})^2 \big[ p_{n-1}^2 + (a_{n-1}^{1 - \gamma})^6 + (a_{n-1}^{1 - \gamma})^2 q_{n-1}^2 \big].
\end{equation}
Recall that $\gamma = 7/6$ and note that $\useconstant{prec}$ does not depend on the choice of $a_0 \geq 288^6$.
\end{lemma}

\begin{proof}
We take $x_1 \in \mathcal{H}^{i_1}_n$ and $x_2 \in \mathcal{H}^{i_2}_n$ and $D_1$, $D_2$ as in Lemma~\ref{l:EA1A2}. Applying Lemma~\ref{l:muL} with $s = a_{n-1}$ and $r = \text{dist}(D_1, D_2)$ (which by \eqref{e:a_lower0} is greater or equal to ${a_n}/{10}$).
We then deduce that,
\begin{equation}
\label{e:muL}
\mu \big( L_{D_1,D_2} \big) \leq c \Big(\frac{a_{n-1}}{a_n} \Big)^2 = c \big(a_{n-1}^{1-\gamma}\big)^2 \; (\leq c), \text{ for all $n \geq 1$.}
\end{equation}

Recalling that $\mathbb{P}_u$ is a Poisson point process in $\mathbb{L}$ having intensity measure $u\mu$, we can infer that
\begin{equation}
\label{e:mu2}
\mathbb{P}_u \big[ 1 \leq \omega(L_{D_1,D_2}) \leq 2 \big] \leq \mathbb{P}_u \big[\omega(L_{D_1, D_2}) \geq 1 \big] \leq u \mu (L_{D_1,D_2})
\end{equation}
and
\begin{equation}
\label{e:mu3}
\begin{split}
\mathbb{P}_u \big[\omega(L_{D_1, D_2}) \geq 3 \big] & \leq \exp\{u \mu ( L_{D_1,D_2})\} - 1 - u\mu ( L_{D_1,D_2}) -\frac{u^2\mu ( L_{D_1,D_2})^2}2\\
& \leq c u^3 \mu ( L_{D_1,D_2})^3, \text{ for all $n \geq 1$,}
\end{split}
\end{equation}
where in the last inequality we used a Taylor expansion together with $u \leq 1$ and with the fact that $u \mu ( L_{D_1,D_2}) \leq c$ by \eqref{e:muL}.

We now use Lemma~\ref{l:EA1A2}, together with \eqref{e:muL} and the two above bounds to obtain (recalling that $u \leq 1$) that:
\begin{equation}
\label{e:EAA}
\mathbb{E}_u \left[ A_{n-1} (x_1) A_{n-1} (x_2) \right] \leq   p_{n-1}^2 + c(a_{n-1}^{1 - \gamma})^6 + c (a_{n-1}^{1- \gamma})^2 q_{n-1}^2, \text{ for $n\geq1$}.
\end{equation}
Taking the supremum over $x_1 \in \mathcal{H}^{i_1}$ and $x_2 \in \mathcal{H}^{i_2}$ as in Lemma~\ref{l:pbound}, this leads to \eqref{e:p_rec}, concluding the proof of Lemma~\ref{l:p_rec}.
\end{proof}

In order to analyze the decay of $q_n(u)$, we will make use of a suitable property of the sets $\mathcal{H}^i_n$ which we now discuss.
Roughly speaking, given two cylinders $C_1 = C(l_1)$ and $C_2 = C(l_2)$ with $l_1$ and $l_2$ as in as in the definition of $q_n(u)$ (see \eqref{e:qn}), we need to bound the number of points in $\mathcal{H}^i_n$ that these cylinders may approach.
Intuitively speaking, the only way in which a given cylinder $C$ may approach too many points in $\mathcal{H}^i_n$ is if $C$ is `approximately tangent' to the sphere $\partial S ( x_0, {(i+1)a_n}/{6})$, see Figure~\ref{f:Hin}.
However, a given cylinder can only be `approximately tangent' to at most one of these spheres, say for some $\bar i \in \{1, \dots, 4\}$.
This explains why we allow $i$ to assume four different values: If we are given two cylinders $C_1$ and $C_2$ as above, we can still find $i_1$ and $i_2$ for which both $C_1$ and $C_2$ are `secant' to the spheres corresponding to $i_1$ and $i_2$.
Consequently, the number of balls in the covering corresponding to $\mathcal{H}^{i_1}_n$ and $\mathcal{H}^{i_2}_n$ that are intersected by $C_1$ and $C_2$ is bounded by an universal constant.
This is made precise in the following

\begin{lemma}
 \label{l:cove_3}
\newconstant{secant} There exists a constant $\useconstant{secant} > 0$ such that the following holds. For any $(a_n)_{n \geq 0}$ as in \eqref{e:new_scales} if we let $\mathcal{H}_n^i$ be defined as in (\ref{e:Hin}), then, for all $n \geq 1$ and every pair of cylinders $C_1, C_2 \in \mathcal{C}$, there exist distinct $i_1, i_2 \in \{1, \ldots , 4\}$ such that
\begin{equation}
\label{eq:cove_3}
\Big| \big\{ x \in \mathcal{H}_n^{i_j} ; ~ S \left( x, \tfrac{a_{n-1}}{10} \right) \times [0,1000] \cap \big( C_1 \cup C_2 \big) \neq \emptyset \big\} \Big| \leq \useconstant{secant},
\end{equation}
for any $j=1,2$. Note that, in accordance with our convention on constants, $\useconstant{secant}$ does not depend on the specific choice of the scale parameter $a_0$.
\end{lemma}

\begin{proof}The proof of this lemma is postponed to the Appendix. \end{proof}

Our next aim is to obtain a recursive equation for $q_n$'s, which resembles the one obtained in \eqref{e:p_rec} for the $p_n$'s. In order to do this, we first establish a result analogous to the Lemmas~\ref{l:pbound} and \ref{l:EA1A2}.

\begin{lemma}
\label{l:Adelta}
Fix $x_0 \in \hex$, $n \geq 1$ and $l_1, l_2 \in \mathbb{L}$. 
Given $i_1$ and $i_2$ as in Lemma~\ref{l:cove_3} we have that
\begin{equation}
\label{e:Adelta}
\begin{split}
\mathbb{E}_u [A_n (x_0, \omega & + \delta_{l_1} + \delta_{l_2})] \leq \useconstant{entrop}^2\Big( \frac{a_n}{a_{n-1}}\Big)^2 \sup_{{x_1 \in \mathcal{H}^{i_1}_n}, \;  {x_2 \in \mathcal{H}^{i_2}_n}} \mathbb{E}_u \left[ A_{n-1} (x_1) A_{n-1} (x_2) \right]\\
& + \useconstant{secant} \; \useconstant{entrop} \Big( \frac{a_n}{a_{n-1}}\Big) \sup_{{x_1 \in \mathcal{H}^{i_1}_n}, \;  {x_2 \in \mathcal{H}^{i_2}_n}}
\mathbb{E}_u \Big[ \begin{aligned}[t]
& A_{n-1} (x_1) A_{n-1} (x_2, \omega + \delta_{l_1} + \delta_{l_2})\\
& + A_{n-1} (x_1, \omega + \delta_{l_1} + \delta_{l_2}) A_{n-1} (x_2) \Big]
\end{aligned}\\
& + \useconstant{secant}^2 \sup_{{x_1 \in \mathcal{H}^{i_1}_n}, \;  {x_2 \in \mathcal{H}^{i_2}_n}} \mathbb{E}_u \big[ A_{n-1} (x_1, \omega + \delta_{l_1} + \delta_{l_2}) A_{n-1} (x_2, \omega + \delta_{l_1} + \delta_{l_2}) \big]
\end{split}
\end{equation}
\end{lemma}

Before going into the proof of Lemma~\ref{l:Adelta}, let us briefly discuss the meaning of the three above terms. Roughly speaking these  terms respectively represent the cases where the lines $l_1$ and $l_2$ influence: `none', `one' or `both' random variables $A_{n-1}(x_1)$ and $A_{n-1}(x_2)$. What is important to notice in these three terms is that, although bounding their corresponding expectations gets harder and harder (as the influence of $l_1$ and $l_2$ increases), the combinatorial factors multiplying these expectations are getting smaller. This trade-off was made possible by Lemma~\ref{l:cove_3} as we will see in the proof below.

\begin{proof}
Recall from Lemma~\ref{l:cove_1} that
\begin{equation}
\partial S \big(x_0, \big(\tfrac{i+1}{6}\big) a_n \big) \subseteq {\textstyle \bigcup \limits_{x \in \mathcal{H}^{i_j}_n}} S \big(x,\tfrac{a_{n-1}}{10} \big), \text{ for $j = 1,2$.}
\end{equation}
Since any path in $\mathbb{R}^2$ connecting the ball $S(x_0, a_n/10)$ to $\partial S(x_0, a_n)$ must intersect both $S (x_1, a_{n-1}/10)$ and $S (x_2, a_{n-1}/10)$, for some $x_1 \in \mathcal{H}^{i_1}_n$ and $x_2 \in \mathcal{H}^{i_2}_n$, it follows that
\begin{equation}
\label{e:EAdelta}
\mathbb{E}_u [A_n (x_0, \omega + \delta_{l_1} + \delta_{l_2})] \leq \sum_{\substack{x_1 \in \mathcal{H}^{i_1}_n, \\  x_2 \in \mathcal{H}^{i_2}_n}} \mathbb{E}_u \big[ A_{n-1} (x_1, \omega + \delta_{l_1} + \delta_{l_2}) A_{n-1} (x_2, \omega + \delta_{l_1} + \delta_{l_2}) \big].
\end{equation}

The lines $l_1$ and $l_2$ may or not influence the functions $A_{n-1}(\cdot)$ appearing above. To distinguish these cases, we defined the sets $\mathcal{D}^{j}_n (x_0, l_1,l_2) \subseteq \mathcal{H}^{i_j}_n$ for $j = 1,2$ by
\begin{equation}
\mathcal{D}^{j}_n (l_1,l_2) = \big\{x \in \mathcal{H}^{i_j}_n; S(x, a_{n-1}/10) \times [0,1000] \cap \big(C(l_1) \cup C(l_2) \big) \not = \varnothing \big\}.
\end{equation}
It is clear from the definitions of $\mathcal{D}^{j}_n$ and $A_n(x, \omega)$ that
\begin{equation}
\label{e:A_no_d}
A_{n-1} (x, \omega + \delta_{l_1} + \delta_{l_2}) = A_{n-1} (x, \omega), \text{ for all $x \in \mathcal{H}^{i_j}_n \setminus \mathcal{D}^{j}_n$.}
\end{equation}

Note that by our choice of $i_1$ and $i_2$ as in Lemma~\ref{l:cove_3}, we have that $|\mathcal{D}^{j}_n| \leq \useconstant{secant}$ for $j = 1,2$ and all $n \geq 1$.
By Lemma~\ref{l:cove_2}, $|\mathcal{H}^{i_j}_n| \leq \useconstant{entrop}(a_n/a_{n-1})$ for every $n \geq 1$.
Now, by splitting the sum in  \eqref{e:EAdelta} into the terms where $x_j \in \mathcal{H}^{i_j}_n \setminus \mathcal{D}^{j}_n$ and $x_j \in \mathcal{D}^{j}_n$ ($j=1,2$) we obtain from \eqref{e:A_no_d} the terms in the right-hand side of \eqref{e:Adelta}, finishing the proof of Lemma~\ref{l:Adelta}.
\end{proof}

We now obtain the promised recurrence relation for $q_n$ analogous to \eqref{e:p_rec}.

\begin{lemma}
\label{l:q_rec}
\newconstant{qrec}
There exists a positive constant $\useconstant{qrec}$ such that, for any $(a_n)_{n \geq 0}$ as in \eqref{e:new_scales} and for all $n \geq 1$,
\begin{equation}
\label{e:q_rec}
\begin{split}
q_n \leq \; & \useconstant{qrec} (a^{\gamma-1}_{n-1})^2 \big[ p^2_{n-1} + (a^{1-\gamma}_{n-1})^6 + (a^{1-\gamma}_{n-1})^2 q^2_{n-1} \big]\\
& + \useconstant{qrec} a_{n-1}^{\gamma-1} \big[ p_{n-1}q_{n-1} + (a^{1-\gamma}_{n-1})^2 q_{n-1} + (a_{n-1}^{1-\gamma})^6 \big] + \useconstant{qrec} \big[ q^2_{n-1} + (a^{1-\gamma}_{n-1})^2\big].
\end{split}
\end{equation}
\end{lemma}

\begin{proof}
We first fix arbitrarily $x_0 \in \hex$, $l_1, l_2$ in $\mathbb{L}$ and we take $i_1$ and $i_2$ as in Lemma~\ref{l:cove_3}. The three terms in the above equation will be derived from the corresponding terms in \eqref{e:Adelta}, after taking the supremum.
Note that the first term in the right-hand side of \eqref{e:Adelta} can be easily bounded using \eqref{e:EAA}, yielding the first term in \eqref{e:q_rec}.

In order to bound the second term in the right-hand side of \eqref{e:Adelta}, we now show that, for every $x_1 \in \mathcal{H}^{i_1}_n$ and $x_2 \in \mathcal{H}^{i_2}_n$,
\begin{equation}
\label{e:EAAd}
\mathbb{E}_u \big[ A_{n-1} (x_1, \omega) A_{n-1} (x_2, \omega + \delta_{l_1} + \delta_{l_2}) \big] \leq c \big[ p_{n-1}q_{n-1} + (a^{1-\gamma}_{n-1})^2 q_{n-1} + (a^{1-\gamma}_{n-1})^6 \big],
\end{equation}
for all $n \geq 1$.

Indeed, using the same notation as in Lemma~\ref{l:EA1A2} and ideas analogous to those in its proof, we obtain first that
\begin{equation}
\label{e:EAAd1}
\begin{split}
\mathbb{E}_u & \big[ A_{n-1} (x_1) A_{n-1} (x_2, \omega + \delta_{l_1} + \delta_{l_2}), \omega(L_{D_1,D_2}) = 0 \big]\\
& \leq \mathbb{E}_u \Big[ A_{n-1} \big(x_1, \mathbf{1}_{L_{D_1} \setminus L_{D_2}} \cdot \omega \big) A_{n-1} \big(x_2, \mathbf{1}_{L_{D_2} \setminus L_{D_1}} \cdot \omega + \delta_{l_1} + \delta_{l_2} \big) \Big]\\
& = \mathbb{E}_u \Big[ A_{n-1} \big(x_1, \mathbf{1}_{L_{D_1} \setminus L_{D_2}} \cdot \omega \big) \Big]
\mathbb{E}_u \Big[ A_{n-1} \big(x_2, \mathbf{1}_{L_{D_2} \setminus L_{D_1}} \cdot \omega + \delta_{l_1} + \delta_{l_2} \big) \Big]\\
& \leq p_{n-1}q_{n-1}.
\end{split}
\end{equation}
Furthermore, we have that
\begin{equation}
\label{e:EAAd2}
\begin{split}
\mathbb{E}_u & \big[ A_{n-1} (x_1) A_{n-1} (x_2, \omega + \delta_{l_1} + \delta_{l_2}), \omega(L_{D_1,D_2}) \geq 3 \big]\\
& \leq \mathbb{P}_u \big[ \omega(L_{D_1,D_2}) \geq 3 \big] \leq c(a^{1-\gamma}_{n-1})^6,
\end{split}
\end{equation}
where the last inequality follows from \eqref{e:mu3} and \eqref{e:muL}.

Finally, we consider
\begin{equation}
\label{e:EAAd3}
\begin{split}
\mathbb{E}_u & \big[ A_{n-1} (x_1) A_{n-1} (x_2, \omega + \delta_{l_1} + \delta_{l_2}), 1\leq \omega(L_{D_1,D_2}) \leq 2 \big]\\
& \leq \mathbb{E}_u \Big[ A_{n-1} \big(x_1, \mathbf{1}_{L_{D_1} \setminus L_{D_2}} \cdot \omega + \mathbf{1}_{L_{D_1,D_2}} \cdot \omega\big) , 1\leq \omega(L_{D_1,D_2}) \leq 2 \Big]\\
& = \mathbb{E}_u \Big[ \mathbf{1}_{\{1\leq \omega(L_{D_1,D_2}) \leq 2\}} \mathbb{E}_u \big[ A_{n-1} \big(x_1, \mathbf{1}_{L_{D_1} \setminus L_{D_2}} \cdot \omega + \mathbf{1}_{L_{D_1,D_2}} \cdot \omega\big) \big| \mathbf{1}_{L_{D_1,D_2}} \cdot \omega \big] \Big]\\
& \leq \sup_{l, l' \in \mathbb{L}} \mathbb{E}_u \big[ A_{n-1}(x_1, \omega + \delta_l + \delta_{l'}) \big] \cdot \mathbb{P}_u \big[ 1\leq \omega(L_{D_1,D_2}) \leq 2 \big] \leq c (a^{1-\gamma}_{n-1})^2 q_{n-1},
\end{split}
\end{equation}
where the last inequality follows from the definition of $q_n$, together with \eqref{e:mu2} and \eqref{e:muL}. Putting together \eqref{e:EAAd1}, \eqref{e:EAAd2} and \eqref{e:EAAd3} we obtain \eqref{e:EAAd} as promised.

To finish the proof, we show that for every $x_1 \in \mathcal{H}^{i_1}_n$ and $x_2 \in \mathcal{H}^{i_2}_n$,
\begin{equation}
\label{e:EAdAd}
\mathbb{E}_u \big[ A_{n-1} (x_1, \omega + \delta_{l_1} + \delta_{l_2}) A_{n-1} (x_2, \omega + \delta_{l_1} + \delta_{l_2}) \big] \leq c [q^2_{n-1} + (a^{1-\gamma}_{n-1})^2], \text{ for all $n\geq 1$,}
\end{equation}
which will correspond to the last term in the right-hand side of \eqref{e:q_rec}.

To prove the above, we first proceed as in \eqref{e:EAAd1} to obtain that
\begin{equation}
\label{e:EAdAd1}
\mathbb{E}_u \big[ A_{n-1} (x_1, \omega + \delta_{l_1} + \delta_{l_2}) A_{n-1} (x_2, \omega + \delta_{l_1} + \delta_{l_2}), \omega(L_{D_1,D_2}) = 0 \big] \leq q^2_{n-1}.
\end{equation}
Then we use \eqref{e:mu2} and \eqref{e:muL} to get
\begin{equation}
\label{e:EAdAd2}
\begin{split}
\mathbb{E}_u & \big[ A_{n-1} (x_1, \omega + \delta_{l_1} + \delta_{l_2}) A_{n-1} (x_2, \omega + \delta_{l_1} + \delta_{l_2}), \omega(L_{D_1,D_2}) \geq 1 \big]\\
& \leq \mathbb{P}_u [\omega(L_{D_1,D_2}) \geq 1] \leq c(a^{1-\gamma}_{n-1})^2.
\end{split}
\end{equation}
Together with \eqref{e:EAdAd1} and \eqref{e:EAdAd2}, this yields \eqref{e:EAdAd}.

The bounds \eqref{e:EAAd} and \eqref{e:EAdAd} can be plugged into \eqref{e:Adelta} (see also the first paragraph of this proof) and, after taking the supremum over $x_0 \in \hex$ and $l_1, l_2 \in \mathbb{L}$, we obtain \eqref{e:q_rec}. This finishes the proof of Lemma~\ref{l:q_rec}.
\end{proof}

Now that we have established a system of relations between $p_n(u)$ and $q_n(u)$, we can obtain the promised induction step used to bound these probabilities. From now on, we finally fix the scale parameter
\begin{equation}
\label{e:a0choice}
\hat{a}_0 = 288^6 \vee \big(8 \; (\useconstant{prec} \vee \useconstant{qrec}) \big)^{168},
\end{equation}
we reefer to Lemmas~\ref{l:p_rec} and \ref{l:q_rec} for the definitions of these constants. As we stressed then, these constants do not depend on the value of $a_0 \in [288^6, \infty)$.

\begin{proposition}
\label{prop:pnpn}
\newconstant{induc}
Take $\hat{a}_0$ as in \eqref{e:a0choice} and define the corresponding sequence $(\hat{a}_n)_{n\geq 1}$ through \eqref{e:gamma}. Then, if for some $n \geq 1$,
\begin{equation}
\label{e:indhyp}
p_{n-1} \leq \hat{a}_{n-1}^{(5/2) (1-\gamma)} \quad \text{and} \quad q_{n-1} \leq \hat{a}_{n-1}^{(3/2) (1-\gamma)},
\end{equation}
then this also holds with $n-1$ replaced by $n$, i.e.
\begin{equation}
p_{n} \leq \hat{a}_{n}^{(5/2) (1-\gamma)} \quad \text{and} \quad q_{n} \leq \hat{a}_{n}^{(3/2) (1-\gamma)}.
\end{equation}
\end{proposition}

\begin{proof}
Using Lemmas~\ref{l:p_rec} and \ref{l:q_rec} together with \eqref{e:indhyp}, we obtain that
\begin{equation}
\begin{split}
p_n & \leq \useconstant{prec} \hat{a}_{n-1}^{2(\gamma-1)} \big[ \hat{a}_{n-1}^{5(1-\gamma)} + \hat{a}_{n-1}^{6(1-\gamma)} + \hat{a}_{n-1}^{2(1-\gamma)}\hat{a}_{n-1}^{3(1-\gamma)} \big] \leq 8 \; \useconstant{prec} \hat{a}_{n-1}^{3(1-\gamma)}\\
q_n & \leq \useconstant{qrec} \hat{a}_{n-1}^{2(\gamma-1)} \big[ \hat{a}_{n-1}^{5(1-\gamma)} + \hat{a}_{n-1}^{6(1-\gamma)} + \hat{a}_{n-1}^{2(1-\gamma)}\hat{a}_{n-1}^{3(1-\gamma)} \big] \\
& \quad + \useconstant{qrec} \hat{a}_{n-1}^{(\gamma-1)} \big[ \hat{a}_{n-1}^{4(1-\gamma)} + \hat{a}_{n-1}^{(7/2)(1-\gamma)} + \hat{a}_{n-1}^{6(1-\gamma)} \big] + \useconstant{qrec} \big[ \hat{a}_{n-1}^{3(1-\gamma)} + \hat{a}_{n-1}^{2(1-\gamma)}\big]\\
& \leq 8 \; \useconstant{qrec} \hat{a}_{n-1}^{2(1-\gamma)}.
\end{split}
\end{equation}

Using \eqref{e:a0choice}, we obtain that
\begin{equation}
\label{e:pnqnan}
\begin{split}
p_n & \leq \hat{a}_0^{1/168} \cdot \hat{a}_{n-1}^{3(1-\gamma)} \overset{\eqref{e:an_ratio}}{=} \hat{a}_0^{1/168} \cdot \hat{a}_{n}^{3(\frac{1}{\gamma} -1)} \leq \hat{a}_n^{3(\frac{1}{\gamma} - 1) + 1/168}\\
q_n & \leq \hat{a}_0^{1/168} \cdot \hat{a}_{n-1}^{2(1-\gamma)} \leq \hat{a}_n^{2(\frac{1}{\gamma} - 1) + 1/168}.\\
\end{split}
\end{equation}

But we have chosen $\gamma$ to be $7/6$, so that
\begin{equation}
3\Big(\frac1\gamma - 1\Big) + \frac{1}{168} < \frac52 (1-\gamma) \quad \text{and} \quad 2\Big(\frac1\gamma - 1\Big) + \frac{1}{168} < \frac32 (1-\gamma),
\end{equation}
which together with \eqref{e:pnqnan} yields the Lemma.
\end{proof}

It is clear from Proposition~\ref{prop:pnpn} that all we need to do in order to obtain the decay of $p_n$ is to bound the values of $p_0(u)$ and $q_0(u)$. This will be done in the next section.

\section{Triggering the recurrence relation}
\label{s:trigger}

In this section we are going to show that for $u$ small enough we have that $p_0 (u) \leq \hat{a}_0^{(5/2) (1-\gamma)}$ and $q_0 (u) \leq \hat{a}_0^{(3/2) (1-\gamma)}$, allowing us to trigger the chain of inequalities provided by Lemma~\ref{prop:pnpn}.
This result will follow once we prove the following

\begin{proposition}
\label{prop:pq_tozero}
As $u$ goes to zero, both $p_0(u)$ and $q_0(u)$ vanish.
\end{proposition}

\begin{remark}
\label{r:diffHR2}
Before going into the proof of this lemma, we would like to stress that this is the first part of the proof of Theorem~\ref{t:main} where our specific choice of the surface $H$ will play a role.
If we had chosen $H$ to be equal to $\mathbb{R}^2$, all the considerations of previous sections would still hold true.
However we will strongly use the rough shape of $H$ in establishing that $q_n(u)$ goes to zero with $u$ (which would not be the case for instance for $\mathbb{R}^2$).
\end{remark}

Recall that the quantity $q_0$ involves the addition of two cylinders $C(l_1)$ and $C(l_2)$ to the random set of cylinders $\mathcal{L}(\omega)$.
In order to easy our analysis we start by proving a lemma which reduces questions concerning two cylinders to a question related to a single (thicker) one.
As before, $\mathbb{C}$ stands for the set of all cylinders of radius one, \textit{i{.}e{.}}, $\mathbb{C}=\{C(l);~l \in \mathbb{L}\}$.

\begin{lemma}
 \label{l:two_to_one_cyl}
Given any two cylinders $C_1, C_2, \in \mathbb{C}$, and any curve $\eta:[0,1] \rightarrow C_1 \cup C_2$, such that dist$(\eta(0), \eta(1)) \geq a_0/10$ we can find a line $l \in \mathbb{L}$ and $t_1 < t_2 \in [0,1]$ such that:
\begin{equation}
\label{e:dist_eta}
\text{dist} (\eta(t_1), \eta(t_2)) \geq a_0/100
\end{equation}
\begin{equation}
\label{e:trace_eta}
 \eta([t_1, t_2]) \subset B(l,4)
\end{equation}
\end{lemma}
\begin{proof}
Let $U = \eta([0,1])$ be the image of the curve $\eta$.

We start by treating two simple cases, namely the one in which $C_1 \cap C_2 = \emptyset$ and the other in which $C_1 \cap C_2 \neq \emptyset$ but $C_1$ and $C_2$ have their axis parallel to each other.
In the former we have that either $U \subset C_1$ or $U \subset C_2$ so there is nothing to be proved and in the latter we can take $l$ to be the axis of $C_1$, and conclude the proof of the lemma.

So, from now on, we assume that $C_1$ and $C_2$ are not parallel and intersect each other, which implies that:
\begin{equation}
 \label{e:C1_C2_nonempty}
C_1 \cap C_2 \text{ is non-empty and bounded.}
\end{equation}

For $x \in \mathbb{R}^3$ and $i=1,2$, we define $\varphi_i(x)$ as the closest point to $x$ lying in the axis of $C_i$.
Note that $\varphi_i ^{-1}(y)$ (for $y$ in the axis of $C_i$) is a plane perpendicular to $C_i$.
From (\ref{e:C1_C2_nonempty}) and the fact that $C_1 \cap C_2$ is convex we conclude that $\varphi_i (C_1 \cap C_2)$ is a line segment whose end-points are denoted by $x_i$ and $y_i$.
Since $U$ is connected and diam$(U) \geq a_0/10$ we have that the set $U \backslash(B(x_1,4) \cup B(y_1,4))$ is non-empty.
Moreover there exists $t_1 < t_2 \in [0,1]$ such that:
\begin{equation}
 \label{e:eta_rest}
\eta([t_1,t_2]) \subseteq U \backslash (B(x_1,4) \cup B(y_1,4)) \text{ and \eqref{e:dist_eta} holds.}
\end{equation}
Indeed, assume that $B(x_1,5)$ is hit by $\eta$ before $B(y_1,5)$ (all the other cases are analogous), then consider the times:
\begin{equation*}
\begin{array}{c}
s_1 \text{ is the time when $\eta$ first enters the ball $B(x_1,5)$},\\
s_2 \text{ is the time of the last visit to $B(x_1,5)$ before visiting $B(y_1,5)$},\\
s_3 \text{ is the time when $\eta$ first enters $B(y_1,5)$} \text{ and}\\  
s_4 \text{ is the  time of the last visit to $B(y_1,5)$}.
\end{array}
\end{equation*}
Then, using the triangle inequality one can see that one of the pairs $(0,s_1)$, $(s_2, s_3)$ or $(s_4, 1)$ satisfy \eqref{e:eta_rest} and \eqref{e:dist_eta}.

Let us denote $U' = \eta([t_1,t_2])$.
Since (\ref{e:dist_eta}) has already been established, all we need to prove is that
\begin{equation}
 \label{e:cyl_radius4}
U' \text{ is contained in a cylinder of radius 4.}
\end{equation}

Let $[x_i,y_i]$ stand for the line segment determined by $x_i$ and $y_i$ and denote by $\bar{C}_i$ the set $C_i \cap \varphi_i^{-1}([x_i,y_i])$.
We now show (\ref{e:cyl_radius4}) by considering the following cases:

\vspace{4mm}

\noindent \emph{Case 1:} $U' \cap (\bar{C}_1 \cup \bar{C}_2) = \emptyset$.

In this case, since $C_1 \cap C_2 \subset \bar{C}_1 \cup \bar{C}_2$ and $U'$ is connected and contained in $C_1 \cup C_2$, we conclude that $U'$ must be entirely contained in one of the two cylinders $C_1$ or $C_2$.
This proves (\ref{e:cyl_radius4}) in this case.

\vspace{4mm}

\noindent \emph{Case 2:} $U' \cap (\bar{C}_1 \cup \bar{C}_2) \neq \emptyset$.

In this case we first show that
\begin{equation}
 \label{e:U'C1C2}
U' \subseteq \bar{C_1} \cup \bar{C}_2
\end{equation}
We start by claiming that
\begin{equation}
 \label{e:Haus1}
\text{Haus}([x_1,y_1],[x_2,y_2]) \leq 2
\end{equation}
where Haus$(\cdot,\cdot)$ stands for the Hausdorff distance between two sets.
Indeed, given $x \in [x_1, y_1]$ (the case $x \in [x_2, y_2]$ is analogous) there is some $z \in C_1 \cap C_2$ such that $\varphi_1(z) = x$.
Then it follows that
\begin{equation*}
\text{dist}(x, [x_{2}, y_{2}]) \leq \text{dist}(x, \varphi_{2}(z)) \leq \text{dist}(x,z) + \text{dist}(z, \varphi_{2}(z)) \leq 2,
\end{equation*}
establishing (\ref{e:Haus1}).

In addition we also claim that
\begin{equation}
 \label{e:Haus2}
\text{Haus}(\{x_1,y_1\},\{x_2,y_2\}) \leq 2 \sqrt{2}.
\end{equation}

In order to prove that, suppose by contradiction that, say, dist$(x_2, \{x_1, y_1\}) > 2 \sqrt{2}$ (the other cases are analogous).
Then we would obtain by (\ref{e:Haus1}), together with Pythagoras' theorem, that $d(\varphi_1(x_2), \{x_1,y_1\}) > 2$.
Since the segment $[x_2,y_2]$ has to be contained in one of the closed half-spaces determined by $\varphi^{-1}_1(\varphi_1(x_2))$, we would have that either dist$(x_1, [x_2,y_2])$ or dist$(y_1, [x_2,y_2])$ would be strictly bigger than $2$, contradicting then (\ref{e:Haus1}).
This establishes the bound (\ref{e:Haus2}).

Recall that $U'$ is connected, contained in $C_1 \cup C_2$ and it intersects $\bar{C}_1 \cup \bar{C}_2$.
Then either \eqref{e:U'C1C2} holds or $U'$ intersects one of the discs $\varphi^{-1}_1 (x_1) \cap C_1$, $\varphi^{-1}_1 (y_1) \cap C_1$, $\varphi^{-1}_2 (x_2) \cap C_1$, and $\varphi^{-1}_2 (y_2) \cap C_2$.
However $U'$ cannot intersect any of those discs since it is, by definition, disjoint from the balls $B(x_1,4) \cup B(y_1,4)$, which, by (\ref{e:Haus2}) contains all these four discs.
It follows that (\ref{e:U'C1C2}) holds.
Since, by (\ref{e:Haus1}) we have that $\bar{C}_1 \cup \bar{C}_2 \subset B([x_1,y_1],4)$ we have that $U'$ is contained in the cylinder of radius $4$ and having the same axis as $C_1$.
This finishes the proof of (\ref{e:cyl_radius4}) in \emph{Case 2} yielding thus the lemma.
\end{proof}

\begin{proof}[Proof of Proposition  \ref{prop:pq_tozero}]
For any $x_0 \in \hex$, we have that $\mu(L_{S(x_0,\hat{a}_0)\times[0,1000]}) < \infty$.
Therefore, as $u$ goes to zero, $\sup_{x_0 \in \hex} \mathbb{P}_u [\omega(L_{S(x_0, \hat{a}_0)\times [0,1000]})>1]$ vanishes.
In particular this already gives that $\lim_{u \to 0} p_0(u) = 0$.
Moreover it also implies that, in order to prove that $\lim_{u \to 0} q_0(u) = 0$, all we have to do is to show that,
\begin{equation}
\label{e:smallpieces}
\begin{array}{c}
\text{for all possible choices of $C_1$ and $C_2$ in $\mathbb{C}$ and $x_0 \in \hex$,}\\
S(x_0, \hat{a}_0/10) \text{ is not connected to } \partial{S}(x_0,\hat{a}_0) \text{ through } \pi((C_1 \cup C_2) \cap H).
\end{array}
\end{equation}
For proving this, it will be convenient to use Lemma \ref{l:two_to_one_cyl}, as we show below.

Assume by contradiction that, for some $x_0 \in \hex$ and cylinders $C_1, C_2 \in \mathbb{C}$, there is a path in $\pi((C_1 \cup C_2)\cap H)$ connecting $S(x_0, \hat{a}_0/10)$ to $\partial S(x_0, \hat{a}_0)$.
This would imply that there would exist some curve $\eta : [0,1] \to (C_1 \cup C_2) \cap H$ for which dist$(\eta(0),\eta(1)) \geq \hat{a}_0/10$.
Using Lemma \ref{l:two_to_one_cyl} we obtain $t_1 < t_2 \in [0,1]$ such that dist$(\eta(t_1),\eta(t_2))\geq \hat{a}_0/100$ and $\eta ([t_1, t_2])$ is contained in $B(l,4)$ for some $l \in \mathbb{L}$.
Denoting again by $\varphi(x)$ the closest point to $x$ in $l$ and using Pythagoras' Theorem we conclude that
\begin{equation}
 \label{e:dpiphieta}
\begin{split}
  \text{dist}(\pi \circ \varphi \circ \eta(t_1), \pi \circ \varphi \circ \eta(t_2)) & \geq \text{dist}(\pi \circ \eta(t_1), \pi \circ \eta(t_2)) - 8 \geq \\
 & \geq \sqrt{\left(\hat{a}_0/100\right)^2 - 1000^2} - 8 \geq 10^7,
\end{split}
\end{equation}
where, in the third inequality, we have used that $\hat{a}_0 > 10^{10}$ which follows directly from  \eqref{e:a0choice}.

Let us denote $\zeta = \pi \circ \varphi \circ \eta$.
Note that $\zeta$ is a linear path defined in $[0,1]$ and taking values in $\mathbb{R}^2$.
By (\ref{e:dpiphieta}) we conclude that dist$(\zeta(t_1), \zeta(t_2)) \geq 10^7$, so we can find $t'_1 < t'_2$ such that the following holds
\begin{gather}
 \label{e:zeta1}
\zeta(t'_1) \text{ and } \zeta(t'_2) \text{ belong to } \mathcal{H},\\
 \label{e:zeta2}
\text{dist}(\zeta(t'_1, t'_2)) \geq 10^6 \text{ and}\\
 \label{e:zeta3}
\zeta([t'_1, t'_2]) \text{ is contained in the segment } [\zeta(t'_1),\zeta(t'_2)] \subset \mathbb{R}^2.
\end{gather}

We now analyze the image of the path $\zeta|_{[{t'_1},{t'_2}]}$ when it is pulled back to $H$.
For that, let us consider the function from $\mathbb{R}^3$ to itself given by $F(x) = (\pi(x), \text{dist}(\pi(x), \mathcal{H}))$ and note that
\begin{gather}
\label{e:propF1}
F \circ \eta = \eta \text{ and }\\
\label{e:propF}
F \text{ is a Lipschitz function with constant } \sqrt{2}.
\end{gather}

It follows that, for all $t \in [0,1]$
\begin{equation}
\begin{split}
 \text{dist}(F(\varphi \circ \eta(t)), \varphi \circ \eta(t)) & \overset{\eqref{e:propF1}}{\leq} \text{dist}(F(\varphi\circ\eta(t)), F(\eta(t))) + \text{dist}(\eta(t), \varphi \circ \eta(t)) \\
 & \overset{\eqref{e:propF}}{\leq} (\sqrt{2} + 1) \text{dist}(\eta(t), \varphi \circ \eta(t)) \leq 10.
\end{split}
\end{equation}

Together with \eqref{e:zeta1}, this inequality implies that $\varphi \circ \eta (t'_i)$, belong to the slab $\mathbb{R}^2 \times [0,10]$ for $i =1,2$.
Since $\varphi \circ \eta$ is contained in the line $l$, using \eqref{e:zeta3} we conclude that $\varphi \circ \eta (t)$ also belongs to the slab $\mathbb{R}^2 \times [0,10]$ for every $t \in [t'_1,t'_2]$.
But that would mean that, the line $\zeta([t'_1, t'_2])$ is contained in $B(\mathcal{H}, 20)$.
This together with \eqref{e:zeta2} leads to a contradiction to the fact that $\mathbb{R}^2 / B(\mathcal{H}, 20)$ is a Lorentz gas with horizon at most $10^4$ (see for instance \cite{Sza00}, p.335) so, $B(\mathcal{H}, 20)$ cannot contain a line segment with length greater than $10^4$.
\end{proof}

\section{Proof of Theorem~\ref{t:main}}
\label{s:main}

In this section we use Lemmas~\ref{prop:pnpn} and \ref{prop:pq_tozero} in order to obtain our main result.
We start by making some considerations concerning the ergodicity properties of the measure $\mathbb{P}_u$.

Consider the space $\{0,1\}^{\mathbb{Q}^3}$ equipped with the canonical sigma-algebra $\mathcal{Y}$ generated by the coordinate projections $(Y_x)_{x \in \mathbb{Q}^3}$ and let $(t_x)_{x \in \mathbb{Q}^3}$ be the shift operator in $\{0,1\}^{\mathbb{Q}^3}$.
Denoting by $\mathcal{Q}_u$ the law of $\psi = \left( \mathbf{1}_{\{x \in \mathcal{V}\}} \right)_{x \in \mathbb{Q}^3}$ on $\{0,1\}^{\mathbb{Q}}$ and by $\Gamma$ the subgroup $2000\cdot\mathbb{Z} e_1$ then, for any event $A \in \mathcal{Y}$ we have (see (3{.}12) in Lemma 3{.}4 in \cite{TW10b})
\begin{equation}
\label{e:ergodic}
\text{if } t_x(A) = A \text{ for all } x \in \Gamma \text{, then } \mathcal{Q}_u [A] \in \{0,1\}.
\end{equation}
 This is to say that $(\mathbb{P}_u \circ \psi^{-1}, (t_x)_{x \Gamma})$ is ergodic.

Let Perc be the event $\{\mathcal{V} \text{ has an unbounded connected component} \}$ and Perc$_H$ the corresponding event replacing $\mathcal{V}$ for $\mathcal{V} \cap H$.
Note that both $\mathbf{1}_{\text{Perc}} \circ \psi^{-1}$ and $\mathbf{1}_{\text{Perc}_H}\circ\psi^{-1}$ are invariant under $(t_x)_{x \in \Gamma}$.
Furthermore, in view of \eqref{e:ergodic} a simple modification in Proposition 3{.}5 in \cite{TW10b} gives that
\begin{equation}
\label{e:PuPerc}
\mathbb{P}_u [\text{Perc}] \in \{0,1\} ~ \text{   and   } ~ \mathbb{P}_u[\text{Perc}_H] \in \{0,1\}
\end{equation}

\begin{remark}
In \cite{TW10b} the authors state the same result corresponding to \eqref{e:ergodic} with $\Gamma = 2000\cdot\mathbb{Z} e_1$ replaced by $\mathbb{Q}{e_1}$, however no modification in their proof is required in order to obtain \eqref{e:ergodic}.
They also prove a result similar to \eqref{e:PuPerc} with the event Perc$_H$ replaced by the event Perc$_2$ where $\mathcal{V}\cap\mathbb{R}^2$ appear instead of $\mathcal{V}\cap H$.
Since the surface $H$ is invariant under $t_x$ for $x\in\Gamma$ the proof that they present adapts easily to our context.
\end{remark}

\begin{proof}[Proof of Theorem~\ref{t:main}]
By Lemma~\ref{prop:pq_tozero}, there is some constant \newconstant{usmall} $\useconstant{usmall} > 0$ such that for all $u \in [0,\useconstant{usmall})$
\begin{equation}
\label{e:p0q0}
p_0(u) \leq \hat{a}_0^{(5/2)(1-\gamma)} \quad \text{and} \quad q_0(u) \leq \hat{a}_0^{(3/2)(1-\gamma)}.
\end{equation}
Then, using Lemma~\ref{prop:pq_tozero} we obtain that for such values of $u$,
\begin{equation}
\label{e:pnbound}
p_n(u) \leq \hat{a}_n^{(5/2)(1-\gamma)}.
\end{equation}
We are now going to see how this implies that
\begin{equation}
\label{e:perc}
\mathbb{P}_u[\mathcal{V} \cap H \text{ has an unbounded connected component}] = 1.
\end{equation}
As above, let the event appearing in \eqref{e:perc} be denoted by $\text{Perc}_H$.
By \eqref{e:PuPerc} it will be enough to prove that $\text{Perc}_H$ has positive probability under $\mathbb{P}_u$.

Let $\mathcal{C}$ be the component of $\mathcal{V}\cap H$ containing the point $(0,0,1000)$.
If the event $\text{Perc}_H^c$ occurs, then $\mathcal{C}$ is bounded.
By the local finiteness of $\omega$, we have that $\mathcal{C}$ is delimited by the intersection of $H$ with a finite number of cylinders in the support of $\omega$.

Note that the intersection of any cylinder with $H$ is a union of pieces of ellipsoids delimited by lines.
This implies that, if $\mathcal{C}$ is non-empty and bounded, then its boundary in $H$ is given by a closed and piecewise smooth curve $\sigma'$ surrounding the point $(0,0,1000)$ in $H$.
We denote by $\sigma$ the projection of the curve $\sigma'$ into $\mathbb{R}^2$ (under the orthogonal map $\pi$).
Clearly, $\sigma$ is a closed curve surrounding the origin in $\mathbb{R}^2$.
This proves that
\begin{equation}
\label{e:Percsigma}
\text{Perc}_H^c \subseteq \Big\{
\begin{array}{c}
\text{there is a closed curve $\sigma \subset \pi(\mathcal{L}\cap H)$}\\
\text{surrounding the origin in $\mathbb{R}^2$}
\end{array}\Big\}
\end{equation}

Let us define the following sequence of real numbers
\begin{equation}
\label{e:xki}
x_{k,i} = \hat{a}_{k-1} + \big(\tfrac{i-1}{10}\big) \hat{a}_{k-1}, \text{ for $k \geq 1$ and $i = 1, \dots, \Big\lceil 10 \big(\tfrac{\hat{a}_k}{\hat{a}_{k-1}} \big) \Big\rceil$}.
\end{equation}
Recall here that $(\hat{a}_n)_{n\geq 1}$ has been obtained in Lemma~\ref{prop:pnpn}.
Note that, using the notation $M_k = \big\lceil 10 \tfrac{\hat{a}_k}{\hat{a}_{k-1}}  \big\rceil$,
\begin{equation}
\label{e:coverray}
[\hat{a}_{k_0-1}, \infty)  \subseteq  \bigcup_{k = k_0}^{\infty} \bigcup_{i=1}^{M_k} S \big(x_{k,i}, \tfrac{\hat{a}_{k-1}}{10} \big), \text{ for every $k_0 \geq 1$.}
\end{equation}

From \eqref{e:Percsigma},
\begin{equation}
\label{e:sigmahit}
\begin{split}
\text{Perc}_H^c \cap \big\{ & \omega \in \Omega; \; \omega(L_{S(0,\hat{a}_{k_0 - 1}) \times [0,1000]}) = 0 \big\}\\
& \subseteq \Big\{
\begin{array}{c}
\text{there is a closed curve $\sigma \subseteq [\pi(\mathcal{L} \cap H)] \setminus S(0,\hat{a}_{k_0-1})$}\\
\text{surrounding the origin in $\mathbb{R}^2$}
\end{array}\Big\}\\
& \subseteq \bigcup_{k = k_0}^{\infty} \Big\{
\begin{array}{c}
\text{there is a closed curve $\sigma \subseteq [\pi(\mathcal{L} \cap H)] \setminus S(0,\hat{a}_{k_0-1})$}\\
\text{surrounding the origin in $\mathbb{R}^2$ and intersecting $[\hat{a}_{k-1}, \hat{a}_k] e_1$}
\end{array} \Big\}.
\end{split}
\end{equation}
Whenever a curve $\sigma$ intersects $[\hat{a}_{k-1}, \hat{a}_k] e_1$, it also intersects one of the balls $S(x_{k,i},\tfrac{\hat{a}_{k-1}}{10})$ for some $i \in 1, \dots, M_k$.
Moreover, such a curve must also intersect the sphere $\partial S(x_{k,i},{\hat{a}_{k-1}})$ in order to surround the origin.
In particular, the event $A_{k-1}(x_{k,i})$ occurs, so that
\begin{equation}
\label{e:Aks2}
\text{Perc}_H^c \cap \big\{ \omega \in \Omega; \; \omega(L_{S(0,\hat{a}_{k_0 - 1}) \times [0,1000]}) = 0 \big\} \subseteq \bigcup_{k = k_0}^{\infty} \bigcup_{i=1}^{M_k}A_{k-1}(x_{k,i}).
\end{equation}
So that
\begin{equation}
\label{e:Aks}
\begin{split}
\mathbb{P}_u \Big[\text{Perc}_H^c, \big\{ \omega & \in \Omega; \; \omega(L_{S(0,\hat{a}_{k_0 - 1}) \times [0,1000]}) = 0 \big\} \Big] \leq \sum_{k = k_0}^{\infty} \sum_{i=1}^{M_k}\mathbb{P}_u[A_{k-1}(x_{k,i})]\\
& \begin{array}{e}
& \overset{\eqref{e:pn}}{\leq} & \sum_{k = k_0}^\infty 20 \big(\tfrac{\hat{a}_{k}}{\hat{a}_{k-1}} \big) p_{k-1}(u) \overset{\eqref{e:pnbound}}{\leq} 20 \sum_{k = k_0}^\infty \hat{a}_{k-1}^{(\gamma-1)} \hat{a}_{k-1}^{(5/2)(1-\gamma)}\\
& = & 20 \sum_{k=k_0}^\infty \hat{a}_{k-1}^{(3/2)(1-\gamma)} = 20 \; \sum_{\mathclap{k=k_0-1}}^\infty \; \big(\hat{a}_0^{-\frac 14} \big)^{\gamma^k}.
\end{array}
\end{split}
\end{equation}
Now recall that
\begin{equation}
\label{e:noak}
\mathbb{P}_u [\omega (L_{S(\hat{a}_{k_0},0) \times [0,1000]}) > 0] \leq u \mu(L_{S(\hat{a}_{k_0},0) \times [0,1000]}).
\end{equation}
Putting \eqref{e:Aks} and \eqref{e:noak} together, we obtain
\begin{equation}
\mathbb{P}_u \big[\text{Perc}_H^c \big] \leq 20 \; \sum_{\mathclap{k=k_0-1}}^\infty \; \big(\hat{a}_0^{-\frac 14} \big)^{\gamma^k} + u \mu(L_{S(\hat{a}_{k_0},0) \times [0,1000]}), \text{ for every choice of $k_0 \geq 1$}.
\end{equation}

Finally, take $k_0$ large enough so that the first term in the right-hand side of the above equation is at most $1/3$ and $u \leq \useconstant{usmall}$ small enough so that the second term is also smaller than $1/3$.
This proves that $\text{Perc}_H$ has positive probability, concluding the proof of Theorem~\ref{t:main}.
\end{proof}

\begin{remark}
\label{r:open}
This paper leaves several questions untouched such as:

1) Is the infinite connected component of $\mathcal{V}$ unique? Note that this is not a direct consequence of the results in \cite{BK89} since the set $\mathcal{L}$ fails to satisfy the so-called finite energy property.
This can be seen from the fact that $\mathcal{L}$ has no bounded components.

2) It is not even clear that the number of infinite connected components of $\mathcal{V}$ belongs almost surely to $\{0,1,\infty\}$, which would be natural to expect from the ergodicity of $\mathcal{V}$, see Lemma~3.3 in \cite{TW10b}.
Note also that the set $\mathcal{V}$ (or a discretized version of it) is not `insertion tolerant' in the sense of \cite{LS99}, Definition~3.2, so that Corollary~3.8 in \cite{LS99} cannot be directly applied in this situation.

3) Assuming the uniqueness of the infinite cluster of $\mathcal{V}$, one could be interested in the decay of the probability that $x$ and $y$ in $\mathbb{R}^d$ are connected through a bounded component of $\mathcal{V}$ as the distance between $x$ and $y$ diverges.
Similar questions have been answered in a rather satisfactory way for the case of Bernoulli percolation, see for instance Theorem~(8.18) of \cite{Gri99}.
\end{remark}

\section*{Acknowledgements}
We thank Johan Tykesson for some helpful discussion about the model.
Part of the collaboration for this work took place during the 2011 summer program at IMPA - Rio de Janeiro.
We acknowledge all the staff for the nice environment provided.
This work had partial financial support from CNPq grant 140532/2007-2 (M{.}R{.}H{.}) and the AXA Research Fund (A{.}T{.}).

\newpage

\appendix
\section*{Appendix}
\label{s:appendix}

\renewcommand{\theequation}{A.\arabic{equation}}
\renewcommand{\thetheorem}{A.\arabic{theorem}}
\setcounter{theorem}{0}
\setcounter{equation}{0}

In what follows, we give the proofs of Lemmas~\ref{l:cove_1}, \ref{l:cove_2} and \ref{l:cove_3}.
They all refer to properties of the sets $\mathcal{H}^i_n$ which were used in previous sections.
Although the proofs of these properties are elementary, we include them here for the readers convenience and specially to emphasize the fact that the constants $\useconstant{entrop}$ and $\useconstant{secant}$ do not depend on the specific choice of the scale parameter $a_0 \geq 288^6$.

Recall the definition of the scale sequence $(a_n)_{n \geq 1}$
\begin{equation}
\label{e:new_scales}
a_0 \in [288^6, \infty), \text{ and } a_n = a^{\gamma^n} \text{for $n \geq 1$},
\end{equation}
where, $\gamma = 7/6$ (as it was fixed in \eqref{e:gamma}).
Recall from \eqref{e:an_ratio} and $a_0 \geq 288^6$ that
\begin{equation}
\label{e:a_lower}
a_n \geq 8000 \quad \text{and} \quad a_{n+1} \geq 288 \, a_{n}, \text{ for every $n \geq 0$}.
\end{equation}

It is important to notice that we consider $a_0$ as a variable. More precisely, all the statements of the lemmas below will hold for any $a_0$ as in \eqref{e:new_scales}. In accordance with our convention on the use of constants, $\useconstant{entrop}$ and $\useconstant{secant}$ are independent of $a_0$.

\begin{proof}[Proof of Lemma~\ref{l:cove_1}]
Take $x_0 \in \hex$ and $y \in \partial S \left( x_0, (i+1)a_n/6 \right)$. By taking some $x \in \mathcal{J}_n$ such that $d(x,y) \leq a_{n-1}/10$, we conclude that $x$ also belongs to $\mathcal{H}_n^i$. This implies that $y$ belongs to $\bigcup_{x \in \mathcal{H}_n^i} $ $S (x, a_{n-1}/10)$, finishing the proof of the lemma.
\end{proof}

\medskip

\begin{proof}[Proof of Lemma~\ref{l:cove_2}]
For a fixed $x_0 \in \hex$, we can split the circle $\partial S \left( x_0, (i+1) a_n / 6 \right)$ into no more than $\lceil 2 \pi \left( {a_n} / {a_{n-1}} \right) \rceil$ arcs having diameter not greater than $a_{n-1}$. Let us denote those arcs by $\{S_k\}_{k=1}^{\lceil 2 \pi \left(a_n/a_{n-1}\right)\rceil}$.

For each $k = 1, \dots, \lceil 2 \pi \left(a_n/a_{n-1}\right)\rceil$, let $B(S_k,4000) = \cup_{x \in S_k} S(x,4000)$. Note that $\{B(S_k, 4000)\}_{k = 1}^{\lceil 2 \pi \left(a_n/a_{n-1}\right)\rceil}$ form a covering of the union $\cup_{x_0 \in \hex} \partial S \left( x_0, (i+1) a_n /6 \right)$ with sets having diameter not larger than $a_{n-1} + 8000$, which by \eqref{e:a_lower} is smaller or equal to $2 a_{n-1}$.

By noting that, for a properly chosen constant $\useconstant{entrop} > 0$ (which is independent of $a_0$), any set $B$ with diameter smaller or equal to $2 a_{n-1}$ can intersect at most $\useconstant{entrop}$ balls in $\{B(x, a_{n-1})\}_{x \in \mathcal{J}_n}$, we finish the proof of the lemma.
\end{proof}

\medskip

\begin{proof}[Proof of Lemma~\ref{l:cove_3}]
For $m =1, 2$ we denote by $B_m$ the orthogonal projection $\pi(C_m)$ of $C_m$ into $\mathbb{R}^2$ and by $l_m$ the projection of the central axis of $C_m$ into $\mathbb{R}^2$.
Generically, $l_m$ will be a line, however if it happens to be a single point, we take $l_m$ to be an arbitrary line containing this point.

The set appearing in (\ref{eq:cove_3}) is contained in $\cup_{m = 1,2} \{x \in \mathcal{H}_n^i; ~ S(x, a_{n-1}/10) \cap B_m \neq \emptyset \}$ and, for each $m \in \{1,2\}$, the corresponding set appearing in this last union
\begin{equation*}
\begin{split}
& \Big\{ x \in \mathcal{J}_n; ~ S(x, \tfrac{a_{n-1}}{10}) \cap B_m \neq \emptyset \text{ and } S(x,a_{n-1}) \cap {\textstyle \bigcup \limits_{x_0 \in \hex}} \partial S \big(x_0, \tfrac{(i+1) a_{n-1}}6 \big) \neq \emptyset \Big\} \\
& \overset{\eqref{e:a_lower}}\subseteq \Big\{x \in \mathcal{J}_n; ~ S(x; a_{n-1}) \text{ intersects both } l_m \text{ and } {\textstyle \bigcup \limits_{x_0 \in \hex}} \partial S \big(x_0, \tfrac{(i+1)a_n}6 \big) \Big\}.
\end{split}
\end{equation*}
We denote by $0$ the origin in $\mathbb{R}^2$. Using again \eqref{e:a_lower}, we conclude that the set above is contained in
\begin{equation}
 \label{eq:cove_3_1}
 \big\{x \in \mathcal{J}_n;~ S(x, 2a_{n-1}) \text{ intersects both } l_m \text{ and } \partial S \big(0, \tfrac{(i+1)a_n}6 \big) \big\}
\end{equation}

Note that if $S(x, 2a_{n-1})$ intersects $\partial S (0, (i+1)a_n/6)$, then it is contained in the annulus
\begin{equation}
\label{eq:annu}
A_n^i := S \left(0, \tfrac{(i+1)a_n}{6}  + 4 a_{n-1} \right) {\Big \backslash} S \left(0, \tfrac{(i+1)a_n}{6} - 4 a_{n-1} \right).
\end{equation}
Thus the set appearing in (\ref{eq:cove_3_1}) is contained in
\begin{equation}
\label{eq:cove_3_2}
\left\{ x \in \mathcal{J}_n ; ~ S(x, 2 a_{n-1}) \cap l_m \cap A_n^i \neq \emptyset \right\}.
\end{equation}
All we have to do now it to bound the size of the set above for well chosen indices $i_1$ and $i_2$ in $\{1,\dots, 4\}$.

In order to analyze the cardinality of the set appearing in (\ref{eq:cove_3_2}) we study the intersection of the line $l_n$ and the annulus $A_n^i$. For that it will be convenient to introduce the distance $d_l$ between a line $l \subset \mathbb{R}^2$ and the origin $0$. For a given $l$, $d_l$ can only belong to at most one of the following intervals:
\begin{equation*}
 \left[\tfrac{3 a_n}{12}, \tfrac{5a_n}{12}\right), \left[\tfrac{5a_n}{12},\tfrac{7a_n}{12}\right), \left[\tfrac{7a_n}{12},\tfrac{9a_n}{12}\right), \text{ and } \left[\tfrac{9a_n}{12},\tfrac{11a_n}{12}\right).
\end{equation*}
Thus we can choose two distinct indices $i_1$ and $i_2$ in $\{1, \ldots, 4\}$ such that
\begin{equation}
 \label{eq:cove_3_3}
 d_{l_m} \notin \left[\tfrac {(2 i_j + 1) a_n}{12}, \tfrac{(2i_j+3)a_n}{12}\right),
\end{equation}
for $m,j = 1,2$.

For fixed $m, j \in \{1,2\}$ we have, of course, two possibilities. Either
\begin{gather}
 \label{eq:cove_3_4}
d_{l_m} \geq \frac{2i_j+3}{12} ~ a_n, \;\; \text{ or}\\
 \label{eq:cove_3_5}
d_{l_m} < \frac{2i_j+1}{12} ~ a_n.
\end{gather}
In the case \eqref{eq:cove_3_4}, $l_m$ does not intersect $S(0, (2i_j+3) a_n /12)$. 
Since, by \eqref{e:a_lower} we have that $A_n^{i_j} \subset S(0, (2i_j +3) a_n/12)$, we can see that the set appearing in \eqref{eq:cove_3_2} is empty and then there is nothing to be proved.

We now turn our attention to the case \eqref{eq:cove_3_5}. 
In this case, $l_m$ intersects the ball $S(0, (2i_j+1) a_n/12)$ and (again by \eqref{e:a_lower}) it must also intersects the inner ball of the annulus $A_n^{i_j}$, hence
\begin{equation}
\label{eq:cove_3_6}
\text{$l_m \cap A_n^{i_j}$ is given by the union of two segments of length } Rf\left(\tfrac{d_{l_m}}{R}, \tfrac{4 a_{n-1}}{R}\right),
\end{equation}
where $R = (i_j + 1) a_n /6$ and $f(x,\epsilon) = (1+\epsilon)\sqrt{(1+\epsilon) - x^2} - (1 - \epsilon) \sqrt{(1-\epsilon) - x^2}$, for all $|\epsilon| \leq 1/24$ and all $|x| \leq 11/12$ (note that $d_{l_m}/R \leq \big((2i_j+1)/{12} \big) ({i_j+1}/{6})^{-1} \leq 11/12$ and, by \eqref{e:a_lower}, $4 a_{n-1} /R \leq 1/24$).

The function $f$ vanishes at $x = 0$ and has bounded derivatives in $[-11/12,11/12] \times [-1/24, 1/24]$. This implies that
\begin{equation}
 \label{eq:cove_3_7}
Rf\left(\tfrac{d_{l_m}}{R}, \tfrac{4 a_{n-1}}{R}\right) \leq R \sup_{ \substack{|x| \leq 11/12 \\ |\epsilon| \leq 1/24}} \left( \frac{\partial f}{\partial \epsilon} (x, \epsilon) \right) ~ \frac{4a_{n-1}}{R} \leq c a_{n-1},
\end{equation}
uniformly for $n \geq 1$.

Using \eqref{eq:cove_3_2} and \eqref{eq:cove_3_6} the bound in \eqref{eq:cove_3_7} implies that
\begin{equation}
\label{eq:cove_3_8}
\begin{split}
& \left| \left\{ x \in \mathcal{H}_n^i ; ~ [0,1000] \times S \left( x, \tfrac{a_{n-1}}{10} \right) \cap C_m \neq \emptyset \right\} \right| \leq \\
& \sup_{h_1, h_2} \big| \left\{ x \in \mathcal{J}_n ; ~S(x, 2 a_{n-1}) \cap (h_1 \cup h_2) \neq \emptyset \right\} \big|,
\end{split}
\end{equation}
where the supremum above is taken over all possible pairs of segments with length smaller or equal to $c a_{n-1}$. 
The proof of the lemma is concluded by observing that the above supremum can be easily bounded by some constant $\useconstant{secant}$.
\end{proof}

\bibliographystyle{plain}
\bibliography{all}

\end{document}